\documentclass[10pt, reqno]{amsart}

\usepackage[utf8]{inputenc}
\usepackage{amsmath}
\usepackage{amssymb}
\usepackage{mathrsfs}
\usepackage{graphics}
\usepackage{geometry}
\usepackage{a4wide}
\usepackage{setspace}
\usepackage{color}
\usepackage{amsthm}
\usepackage{paralist}
\usepackage{dsfont}
\usepackage{verbatim}
\usepackage{bm}
\usepackage[numbers]{natbib}

\newcommand{\dd}{\,\mathrm{d}}

\newcommand{\dv}{\mathrm{div}}

\newcommand{\fder}[1]{\frac{\dd}{\dd #1}}
\newcommand{\rx}{\mathrm{x}}
\newcommand{\ry}{\mathrm{y}}

\newcommand{\auxil}{\mathfrak{A}}
\newcommand{\metr}{\mathbf{d}}

\newcommand{\ent}{\mathcal{W}}

\newcommand{\prb}{\mathscr{P}}
\newcommand{\prbb}{\mathscr{P}}

\newcommand{\ball}{\mathbb{B}}

\newcommand{\eps}{\varepsilon}
\newcommand{\W}{\mathbf{W}}

\newcommand{\m}{\mathfrak{m}}
\newcommand{\N}{{\mathbb{N}}}
\newcommand{\R}{{\mathbb{R}}}
\newcommand{\Rd}{{\mathbb{R}^d}}

\newcommand{\eins}[1]{\mathbf{1}_{#1}}

\newcommand{\Matn}{\R^{n\times n}}
\newcommand{\tT }{\mathrm{T}}

\newcommand{\supp}{\mathrm{supp}\,}
\newcommand{\diam}{\mathrm{diam}\,}
\newcommand{\einsvec}{\mathsf{e}}
\newcommand{\id}{\operatorname{id}}
\newcommand{\mom}[1]{\boldsymbol{\m}_2(#1)}

\newtheorem{definition}{Definition}[section]
\newtheorem{remark}[definition]{Remark}
\newtheorem{lemma}[definition]{Lemma}
\newtheorem{thm}[definition]{Theorem}
\newtheorem{prop}[definition]{Proposition}

\newtheorem{coro}[definition]{Corollary}

\newtheorem*{thm*}{Theorem}
\newtheorem*{coro*}{Corollary}
\newtheorem*{xmp*}{Example}

\begin{document}

\begin{abstract}
We consider a system of $n$ nonlocal interaction evolution equations on $\Rd$ with a differentiable matrix-valued interaction potential $W$. Under suitable conditions on convexity, symmetry and growth of $W$, we prove $\lambda$-geodesic convexity for some $\lambda\in\R$ of the associated interaction energy with respect to a weighted compound distance of Wasserstein type. In particular, this implies existence and uniqueness of solutions to the evolution system. In one spatial dimension, we further analyse the qualitative properties of this solution in the non-uniformly convex case. We obtain, if the interaction potential is sufficiently convex far away from the origin, that the support of the solution is uniformly bounded. Under a suitable Lipschitz condition for the potential, we can exclude finite-time blow-up and give a partial characterization of the long-time behaviour. 
\end{abstract}

\title[System of nonlocal interaction equations]{Geodesically convex energies and confinement of solutions for a multi-component system of nonlocal interaction equations}
\author[Jonathan Zinsl]{Jonathan Zinsl}
\address{Zentrum f\"ur Mathematik \\ Technische Universit\"at M\"unchen \\ 85747 Garching, Germany}
\email{zinsl@ma.tum.de}
\date{\today}
\maketitle


\section{Introduction}\label{sec:intro}
\subsection{The evolution system and its variational structure}
In this work, we analyse the following system of $n\in\N$ nonlocal interaction evolution equations
\begin{align}
\label{eq:pdend}
\begin{split}
\partial_t \mu_1&=\dv[m_1\mu_1 \nabla (W_{11}\ast\mu_1+W_{12}\ast \mu_2+\ldots+W_{1n}\ast\mu_n)],\\
\partial_t \mu_2&=\dv[m_2\mu_2 \nabla (W_{21}\ast\mu_1+W_{22}\ast \mu_2+\ldots+W_{2n}\ast\mu_n)],\\
&\vdots\\
\partial_t \mu_n&=\dv[m_n\mu_n \nabla (W_{n1}\ast\mu_1+W_{n2}\ast \mu_2+\ldots+W_{nn}\ast\mu_n)].
\end{split}
\end{align}
The sought-for $n$-vector-valued solution $\mu(t)=(\mu_1(t),\ldots,\mu_n(t))$ describes the distribution or concentration of $n$ different populations or agents on $\Rd$ at time $t\ge 0$, $d\in\N$ denoting the spatial dimension. Apart from the constant \emph{mobility magnitudes} $m_1,\ldots,m_n>0$, system \eqref{eq:pdend} is mainly governed by the matrix-valued \emph{interaction potential} $W:\,\Rd\to\Matn$ satisfying the following requirements:
\begin{enumerate}[({W}1)]
\item $W(z)$ is a symmetric matrix for each $z\in\R^d$.
\item $W_{ij}\in C^1(\R^d;\R)$ for all $i,j\in \{1,\ldots,n\}$.
\item $W(z)=W(-z)$ for all $z\in\R^d$.
\item There exists a matrix $\overline{W}\in\R^{n\times n}$ such that for each $i,j\in \{1,\ldots,n\}$ and all $z\in\R^d$: $$|W_{ij}(z)|\le \overline{W}_{ij}(1+|z|^2).$$
\item There exists a symmetric matrix $\kappa\in\R^{n\times n}$ such that, for each $i,j\in \{1,\ldots,n\}$, $W_{ij}$ is \break $\kappa_{ij}$-(semi-)convex, i.e. the map $z\mapsto W_{ij}(z)-\frac12\kappa_{ij}|z|^2$ is convex.
\end{enumerate}
System \eqref{eq:pdend} possesses a formal gradient flow structure: On the subspace $\prbb$ of those $n$-vector Borel measures on $\Rd$ with fixed \emph{total masses} $\mu_j(\Rd)=p_j>0$, fixed \emph{(joint, weighted) center of mass} 
\begin{align*}
\sum_{j=1}^n \frac1{m_j}\int_{\Rd}x\dd \mu_j(x)&=E\in\Rd,
\end{align*} 
and finite second moments $\mom{\mu_j}:=\int_\Rd |x|^2\dd\mu_j(x)$, the \emph{multi-component interaction energy} functional
\begin{align}
\label{eq:energy}
\ent(\mu):=\frac12\sum_{i=1}^n\sum_{j=1}^n \int_\Rd\int_\Rd W_{ij}(x-y)\dd \mu_i(x)\dd \mu_j(y)
\end{align}
induces \eqref{eq:pdend} as its gradient flow w.r.t. the following compound metric of Wasserstein-type distances for each of the components of the vector measures $\mu^0,\mu^1\in\prbb$:
\begin{align}
\label{eq:metric}
\W_\prbb(\mu^0,\mu^1)&=\left[\sum_{j=1}^n \frac1{m_j}\inf\left\{\int_{\R^d\times\R^d} |x-y|^2\dd\gamma_j(x,y)\,\bigg.\bigg|\, \gamma_j\in\Gamma(\mu^0_j,\mu^1_j)\right\}\right]^{1/2},
\end{align}
where $\Gamma(\mu^0_j,\mu^1_j)$ is the subset of finite Borel measures on $\R^d\times\R^d$ with marginals $\mu^0_j$ and $\mu^1_j$. It easily follows from the properties of the usual Wasserstein distance for probability measures with finite second moment that $\W_\prbb$ defines a distance on the (geodesic) space $\prb$ (see for instance \cite{villani2003,savare2008} for more details on optimal transport and gradient flows). Note that (W1)-(W5) imply (at least formally) that $\prb$ is a positively invariant set along the flow of the evolution \eqref{eq:pdend}. In this work, we give a rigorous proof for these formal arguments.
\subsection{Main results}
We present our main results and improvements upon the case of a scalar interaction equation ($n=1$).
\subsubsection{Convexity along generalized geodesics and existence of solutions}
We obtain in the case of genuine \emph{irreducible} systems (see Definition \ref{def:irred} below) a novel sufficient condition on the model parameters such that the interaction energy functional $\ent$ becomes $\lambda$-convex along \emph{generalized} geodesics on $\prbb$ with respect to the distance $\W_\prbb$ for some $\lambda\in\R$:
\begin{thm*}[Convexity of $\ent$ and generation of gradient flow solution]
Let $n>1$, assume that \textup{(W1)--(W5)} hold and let \eqref{eq:pdend} be irreducible (see Definition \ref{def:irred}). Define for $i\in \{1,\ldots,n\}$ the quantity $\eta_i:=\min\limits_{j\neq i}\kappa_{ij}m_j\in\R$. Then, $\ent$ is $\lambda$-convex along generalized geodesics on $\prbb$ w.r.t. $\W_\prbb$ (see Definition \ref{definition:generalconv}) for all
\begin{align}
\label{eq:convcond}
\lambda&\le\min_{i\in\{1,\ldots,n\}}\left[p_i\min(0,m_i\kappa_{ii}-\eta_i)+\frac12\sum_{j=1}^n p_j\left(\eta_j+\eta_i\frac{m_i}{m_j}\right)\right].
\end{align}
Furthermore, the system \eqref{eq:pdend}, endowed with an initial datum $\mu^0\in\prb$, induces a $\lambda$-contractive gradient flow \break $\mu\in AC^2_{\mathrm{loc}}([0,\infty);(\mathscr{P},\W_\prbb))$ (see Theorem \ref{thm:exun}).\\
If $\ent$ is uniformly geodesically convex, i.e., $\lambda$-convex along generalized geodesics for some $\lambda>0$, the measure
\begin{align*}
\mu^\infty:=(p_1,\ldots,p_n)^\tT\delta_{x^\infty},\qquad\text{with}\qquad x^\infty:=E\left[\sum_{j=1}^n\frac{p_j}{m_j}\right]^{-1}\in\Rd,
\end{align*}
is the unique minimizer---the \emph{ground state}---of $\ent$ and the unique stationary state of \eqref{eq:pdend} on $\prb$. It is globally asymptotically stable: the gradient flow solution converges exponentially fast in $(\prb,\W_\prbb)$ with rate $\lambda$ to $\mu^\infty$. 
\end{thm*}

\emph{Comparison to the scalar case.~}
Our condition on convexity along generalized geodesics \eqref{eq:convcond} extends the well-known condition by McCann \cite{mccann1997} to multiple components: For $n=1$, if the interaction potential $W$ is $\kappa$-convex on $\Rd$, then $\ent$ is $m\kappa p$-convex along generalized geodesics in the space $\prbb$ defined as above for $n=1$. The novel condition \eqref{eq:convcond} is a \emph{non-straightforward} generalization of McCann's condition to the case of systems $n>1$, especially if one seeks a strictly positive modulus of convexity $\lambda$. On the one hand, requiring that all cross- and self-interactions are attractive (reflected by strictly positive $\kappa_{ij}$ for all $i,j\in\{1,\ldots,n\}$ in assumption (W5)) provides $\lambda>0$ in \eqref{eq:convcond} but is \emph{far from optimal}. On the other hand, the value of $\lambda$ in \eqref{eq:convcond} does not seem to depend (at least not in an obvious manner) on the eigenvalues of the matrix $\kappa\in\R^{n\times n}$ from (W5). Hence, our new convexity condition uncovers---in a quantitative way---the dynamics of a multi-component network of (e.g. particle) interactions driven by $C^1$ potentials. In contrast to \cite{carrillo2011,difranc2013}, we restrict ourselves to the symmetric and smooth framework provided by (W1)--(W5), even if existence and uniqueness of solutions possibly could also be proved in more general settings. Let us now illustrate our novel condition \eqref{eq:convcond} by several examples. Set, for simplicity, $m_j=1$ and $p_j=1$ for all $j\in\{1,\ldots,n\}$. Then, \eqref{eq:convcond} simplifies to
\begin{align*}
\lambda&\le\min_{i\in\{1,\ldots,n\}}\left[\min(0,\kappa_{ii}-\eta_i)+\frac12\sum_{j=1}^n \eta_j+\frac{n}{2}\eta_i\right]\qquad (n>1),
\end{align*}
where here $\eta_i=\min\limits_{j\neq i}\kappa_{ij}$. In the even more specific setting of two species ($n=2$), one has $\eta_1=\kappa_{12}=\eta_2$ and
\begin{align*}
\lambda&\le\min\{\kappa_{11},\kappa_{12},\kappa_{22}\}+\kappa_{12}\qquad(n=2).
\end{align*}
Clearly, if all interactions are attractive, e.g. if
\begin{align*}
\kappa=\begin{pmatrix}2 & 1 \\ 1 & 2\end{pmatrix},
\end{align*}
then $\ent$ is uniformly convex (here, the modulus is equal to $2$) along generalized geodesics. One can also allow for repulsive self-interaction if overall attraction dominates repulsion, e.g. if
\begin{align*}
\kappa=\begin{pmatrix}-1 & 2 \\ 2 & -1\end{pmatrix}.
\end{align*}
Formula \eqref{eq:convcond} yields in this case the $1$-convexity of $\ent$. Note that $\kappa$ is an indefinite matrix. In contrast to that, repulsive cross-interactions, as in the case
\begin{align*}
\kappa=\begin{pmatrix}2 & -1 \\ -1 & 2\end{pmatrix},
\end{align*}
may lead to non-uniform convexity $\lambda\le 0$; here $\lambda\le -2$. Remarkably, $\kappa$ is positive definite.
\subsubsection{Qualitative behaviour for non-uniformly convex energy}
If $\ent$ is $\lambda$-geodesically convex with only $\lambda\le 0$, the dynamics of system \eqref{eq:pdend} are more involved. There, we restrict to one spatial dimension ($d=1$) and rewrite the system in terms of \emph{inverse distribution functions}: Given the (scaled) cumulative distribution functions 
\begin{align}
F_i(t,x)&=\int_{-\infty}^x \frac1{p_i}\dd \mu_i(t,y)\quad\in [0,1],\label{eq:cumF}
\end{align}
let $u_i$ be their corresponding pseudo-inverse, i.e.
\begin{align}
u_i(t,z)&=\inf\{x\in\R:\,F_i(t,x)> z\}\quad (\text{for }z\in [0,1)).\label{eq:invu}
\end{align}
Then, system \eqref{eq:pdend} transforms into (cf. Section \ref{sec:qual})
\begin{align}
\label{eq:pdeinv}
\partial_t u_i(t,z)&=m_i\sum_{j=1}^n p_j\int_0^1 W_{ij}'(u_j(t,\xi)-u_i(t,z))\dd \xi\qquad (i=1,\ldots,n).
\end{align}
In terms of system \eqref{eq:pdeinv}, if $\mu\in AC^2_{\mathrm{loc}}([0,\infty);(\mathscr{P},\W_\prbb))$, one has $u\in AC^2_{\mathrm{loc}}([0,\infty);L^2([0,1];\R^n))$ and for all $t\ge 0$ and all $i\in \{1,\ldots,n\}$, $u_i(t,\cdot)$ is a non-decreasing \emph{càdlàg} function on $(0,1)$. Section \ref{sec:qual} is devoted to the analysis of the qualitative behaviour of the solution $\mu$ to \eqref{eq:pdend} by means of investigation of the corresponding solution $u$ to \eqref{eq:pdeinv}. In that part, our main result is a \emph{confinement} property of the solution: 
\begin{thm*}[Confinement of solutions]
Assume that \textup{(W1)--(W5)} are satisfied and let an initial datum $\mu^0$ with compact support be given. Then, the solution $\mu$ to \eqref{eq:pdend} \emph{propagates at finite speed}: for each $T>0$, there exists a constant $K(T,\mu^0)>0$ depending on $T$ and $\mu^0$ such that
\begin{align*}
\supp\mu_i(t)&\subset [-K(T,\mu^0),K(T,\mu^0)]\quad\forall~t\in [0,T]\quad\forall~i\in\{1,\ldots,n\}.
\end{align*}
Assume now in addition that the system \eqref{eq:pdend} is \emph{irreducible at large distance} (see Definition \ref{def:selfc}) and that the following condition is satisfied: there exist $R>0$ and a matrix $C\in\Matn$ such that for each $i,j\in\{1,\ldots,n\}$, the map $W_{ij}$ is $C_{ij}$-(semi-)convex on the interval $(R,\infty)$ and the following holds: \\ 
If $n=1$, then $C>0$. If $n>1$, with $\tilde\eta_i:=\min\limits_{j\neq i}C_{ij}p_j$ for all $i\in\{1,\ldots,n\}$,
\begin{align*}
\tilde\lambda_0&:=\min_{i\in\{1,\ldots,n\}}\left[p_i\min(0,m_iC_{ii}-\tilde\eta_i)+\frac12\sum_{j=1}^n p_j\left(\tilde\eta_j+\tilde\eta_i\frac{m_i}{m_j}\right)\right]>0.
\end{align*}
Then, the solution $\mu$ to \eqref{eq:pdend} is \emph{uniformly confined}: there exists $K(\mu^0)>0$ such that
\begin{align}
\label{eq:finsp}
\supp\mu_i(t)&\subset [-K(\mu^0),K(\mu^0)]\quad\forall~t\ge 0\quad\forall~i\in\{1,\ldots,n\}.
\end{align}
\end{thm*}
%

Thus, in many cases, mass cannot escape to infinity. In contrast, is it possible to have \emph{concentration} in finite time, i.e. can it occur that absolutely continuous solutions collapse to measures with nonvanishing singular part in finite time? The answer is negative for Lipschitz-continuous $W_{ij}'$ and absolutely continuous initial data with continuous and bounded Lebesgue density (cf. Proposition \ref{prop:noblowup}).

Section \ref{subsec:long} is devoted to the study of the long-time behaviour of the solution to \eqref{eq:pdend}. We first prove (cf. Theorem \ref{prop:ltb}) that if the solution is \emph{a priori} confined to a compact set, the $\omega$-limit set of the system only contains steady states of \eqref{eq:pdend}. More specifically, assume that \eqref{eq:finsp} is true and that all $W_{ij}'$ are Lipschitz-continuous on the interval $[-2K,2K]$. Then,
\begin{align}
\label{eq:entdiss}
\lim_{t\to\infty}\left(\fder{t}\ent(\mu(t))\right)&=0.
\end{align}
Moreover, for each sequence $t_k\to\infty$, there exists a subsequence and a steady state $\overline{\mu}\in\prb$ of \eqref{eq:pdend} such that on the subsequence
\begin{align*}
\lim_{k\to\infty} \W_1(\mu_i(t_k),\overline{\mu}_i)=0\quad\forall\, i\in\{1,\ldots,n\}.
\end{align*}
There, $\W_1$ denotes the $L^1$-Wasserstein distance between finite measures. However, this large-time limit $\overline{\mu}$ is not unique since it depends both on the sequence $(t_k)_{k\in\N}$ chosen and the extracted subsequence. 

Even if the interaction potential does neither yield uniform geodesic convexity of the energy nor is confining and Lipschitz, we may observe a $\delta$\emph{-separation phenomenon}: If the initial datum has compact support and the model parameters admit $\sum_{j=1}^n \kappa_{ij}p_j>0$ for all $i$, the diameter of the support of the solution shrinks exponentially fast over time (cf. Proposition \ref{prop:delta}). Still, the solution does in general not converge to a fixed steady state. However, in the uniformly geodesically convex regime ($\lambda>0$ in \eqref{eq:convcond}), we obtain convergence even w.r.t. the stronger topology of the $L^\infty$-Wasserstein distance $\W_\infty$,
\begin{align*}
\lim_{t\to\infty} \W_\infty(\mu_i(t),\mu^\infty_i)=0\quad\forall\, i\in\{1,\ldots,n\},
\end{align*}
for initial data with compact support. In contrast to convergence w.r.t. $\W_\prbb$ (cf. Corollary \ref{coro:posconv}), we do not obtain a specific rate of convergence.
\\

\emph{Comparison to the scalar case.~}
For the scalar equation $n=1$, an analogous confinement property of the solution to \eqref{eq:pdend} has been derived e.g. in \cite{carrillo2012,balague2014} in arbitrary spatial dimension $d$ under relatively general requirements on the regularity of the interaction potential $W$, assuming that $W$ is (at least in a weak sense) attractive at large distance from the origin. In this work, we investigate the qualitative behaviour of solutions to \eqref{eq:pdend} in one spatial dimension $d=1$ only, concentrating on possible complications arising from the coupling in the equations \eqref{eq:pdend}. We obtain here a multi-component condition guaranteeing that $W$ is confining, which is similar to the conditions from \cite{carrillo2012,balague2014}, but taking into account that not all self- and cross-interactions have to be attractive. Indeed, we require $W$ to behave like a uniformly convex potential in the sense of our new condition \eqref{eq:convcond}. A connatural property is assumed in the scalar case, see \cite{carrillo2012} in view of McCann's condition \cite{mccann1997}. At the present state, it is unclear if the approach from \cite{carrillo2012,balague2014} to prove uniform compactness of the support of the solution $\mu$ to \eqref{eq:pdend} in more than one spatial dimension can be carried over to the case of systems $n>1$. In contrast, using the approach with inverse distribution functions for $d=1$ allows us to give a more complete description of the qualitative behaviour of solutions: exclusion of blow-up in finite time can be shown with the method from \cite{burger2008} also in the case of systems $n>1$. Moreover, our result on the structure of the $\omega$-limit set of the dynamical system associated to \eqref{eq:pdend}, as a generalization of the scalar result from \cite{raoul2012}, also illustrates the use of energy methods, even if $\ent$ is not uniformly convex along generalized geodesics.
\subsection{Background}
System \eqref{eq:pdend} is a natural generalization of the scalar nonlocal evolution equation
\begin{align}
\label{eq:pde1d}
\partial_t \mu&=\dv[m\mu\nabla (W\ast\mu)],
\end{align}
to multiple components. For the corresponding interaction energy functional
\begin{align}
\label{eq:ent1d}
\ent(\mu)&=\frac12\int_\Rd (W\ast \mu)\dd \mu,
\end{align}
McCann provided in his seminal paper \cite{mccann1997} a criterion for $\lambda$-geodesic convexity with respect to the $L^2$-Wasserstein distance (see also \cite[Thm. 5.15(c)]{villani2003}). In a nutshell, if $W$ is $\kappa$-convex in the Euclidean sense on $\R^d$ for some $\kappa\in\R$, then $\ent$ is $\min(0,\kappa)$-geodesically convex on the space of probability measures endowed with the $L^2$-Wasserstein distance. On the subspace of those measures having fixed center of mass, $\ent$ is $\kappa$-geodesically convex (i.e. uniform convexity is retained in the metric framework). It was proven by Ambrosio, Gigli and Savaré in \cite{savare2008} that geodesic convexity essentially leads to existence and uniqueness of weak solutions for the associated gradient flow evolution equation and to contractivity of the associated flow map. Geodesic convexity also yields useful error estimates, e.g. for the semi-discrete JKO scheme \cite{jko1998,otto2001} often used to construct weak solutions to equations with gradient flow structure. An immediate consequence of $\lambda$-geodesic convexity of functionals -- for strictly positive $\lambda\in\R$ -- is existence and uniqueness of minimizers (for recent results without using convexity, see e.g. \cite{carrillo2014,carrillo2014gr}). For more general genuine systems of equations with gradient flow structure, geodesic convexity has been studied in \cite{zinsl2014}.

Model equations of the form \eqref{eq:pde1d} have arised in the study of population dynamics in many cases (e.g. \cite{bcc2012,blanchet2006,burger2007,carrillo2009,colombo2012,kang2009,kolokonikov2013,luckhaus2012,mogilner1999,topaz2004,topaz2006}) often derived as the infinite-particle limit of a individual-based model (e.g. \cite{bodnar2006,golse2003,morale2005}):
\begin{itemize}
\item In the parabolic-elliptic Keller-Segel model for chemotaxis in two spatial dimensions, the interaction potential is given by (the negative of) the Newtonian potential, i.e. $W(z)=\frac1{2\pi}\log(|z|)$, which is singular at $z=0$ and attractive.
\item Typical mathematical models of swarming processes include so-called \emph{attractive-repulsive} potentials of the form $W(z)=-C_a e^{-|z|/l_a}+C_re^{-|z|/l_r}$, a special case of which is the attractive \emph{Morse potential} $W(z)=-e^{-|z|}$. Also, Gaussian-type attractive-repulsive potentials $W(z)=-C_a e^{-|z|^2/l_a}+C_re^{-|z|^2/l_r}$ are conceivable.
\end{itemize}
Nonlocal interaction potentials also appear in several models of physical applications such as models for granular media (\cite{benedetto1997,cmv2003,cmv2006,li2004,toscani2000,carrillo2004}), opinion formation \cite{toscani2006} or interactions between particles (e.g. in crystals \cite{theil2006} or fluids \cite{zhang2009}) with a broad range of reasonable interaction potentials. One can e.g. consider
\begin{itemize}
\item convex and $C^1$-regular potentials, e.g. $W(z)=|z|^q$ for $q>1$.
\item non-convex, but regular potentials such as the \emph{double-well potential} $W(z)=|z|^4-|z|^2$,
\item non-convex and singular potentials, e.g. the \emph{Lennard-Jones} potential.
\end{itemize}
In the case of a radially symmetric potential $W(z)=w(|z|)$, the effect of the interaction potential is reflected by the sign of $w'$: If $w'$ is positive, the individuals of the population \emph{attract} each other, whereas in the case of negative $w'$ the dynamics are \emph{repulsive}. The force generated by the potential $W$ points towards or away from the origin for positive or negative $w'$, respectively. Radially symmetric potentials describe interactions only depending on the distance of the particles. With the sum of convolutions appearing in the flux on the r.h.s. of system \eqref{eq:pdend}, we take into account that every species generates a -- probably long-range -- force on every other species.

Naturally, aggregation processes modelled by nonlocal interaction potentials are often combined with diffusive processes yielding (nonlinear) drift-diffusion equations as mathematical models. The question of global existence of solutions to equations of these forms has been addressed in various publications. Using the theory of gradient flows, global existence of measure-valued solutions was proven in \cite{carrillo2011,carrillo2014non}, also for non-smooth potentials, in generalization of \cite{cmv2003,cmv2006}. Methods from optimal transportation theory were also useful for proving uniqueness, see e.g. \cite{carrillo2010,crippa2013}. Well-posedness in the measure-valued sense was also studied in \cite{canizo2011}, and in \cite{difranc2013} for a similar system for two species (see below). 

A second field of study is the analysis of the qualitative behaviour of solutions to equations like \eqref{eq:pde1d}, such as the speed of propagation, finite- and infinite-time blow-up of solutions and possible attractors, also with focus on self-similarity of solutions. It is not surprising that \eqref{eq:pde1d} exhibits blow-ups if the potential is sufficiently attractive. The aforementioned properties were investigated e.g. in \cite{balague2014,bertozzi2010,bertozzi2009,bertozzi2007,bertozzi2011,biler2009,carrillo2012,laurent2007,sun2012}. One specific object of study is equation \eqref{eq:pde1d} considered in one spatial dimension. For instance, in \cite{raoul2012,fellner2010,fellner2011} by Raoul and Fellner, rewriting in terms of inverse distribution functions allowed for the characterization of the long-time behaviour and the set of possible steady states of \eqref{eq:pde1d}. One-dimensional models with nonlinear diffusion have been studied e.g. by Burger and Di Francesco in \cite{burger2008}.

Genuine systems of the specific form \eqref{eq:pdend} have been studied only in the case of two species as a physical model for two-component mixtures \cite{giessen1999}, fluids \cite{zhang2009} or particle interactions \cite{giacomin2000}. From a more mathematical point of view, it was analysed by Di Francesco and Fagioli in \cite{difranc2013,dif2016}, where the results from \cite{carrillo2011,raoul2012,fellner2010,fellner2011} were generalized to the case of two components. There, also the case of non-symmetric interaction has been studied.


\section{Geodesic convexity and existence of solutions}\label{sec:cvx}
In this section, we derive a sufficient condition for $\lambda$-convexity along generalized geodesics of the interaction energy $\ent$ (cf. formula \eqref{eq:energy}) and conclude existence and uniqueness of solutions to \eqref{eq:pdend}. Throughout this part, the assumptions \textup{(W1)--(W5)} above shall be fulfilled. 

We begin with our definition of convexity along generalized geodesics, which is a straightforward generalization of the respective definition in the scalar case (see \cite[Ch. 9]{savare2008}) to our vector-valued setting:
\begin{definition}[$\lambda$-convexity along generalized geodesics]\label{definition:generalconv}
Given $\lambda\in\R$, we say that a functional \break $\auxil:\prbb\to\R_\infty$ is \emph{$\lambda$-convex along generalized geodesics} on $(\prbb,\W_\prbb)$, if for any triple $\mu^1,\mu^2,\mu^3\in\prbb$, there exists a $n$-vector-valued Borel measure $\bm{\mu}$ on $\R^d\times\R^d\times\R^d$ such that:\vspace{-\topsep}
\begin{itemize}
\item For all $j\in\{1,\ldots,n\}$ and all $k\in\{1,2,3\}$: ${\mu}^k_j=\pi^k{_\#}\bm{\mu}_j$.
\item For $k\in\{2,3\}$ and all $j\in\{1,\ldots,n\}$, the measure $\pi^{(1,k)}{_\#}\bm{\mu}_j$ is optimal in $\Gamma(\mu^1_j,\mu^k_j)$, i.e. it realizes the minimum in
\begin{align*}
\inf\left\{\int_{\R^d\times\R^d}|x^1-x^k|^2\dd \gamma_j(x^1,x^k)\bigg.\bigg|\gamma_j\in\Gamma(\mu^1_j,\mu^k_j)\right\}.
\end{align*}
\item Defining for $s\in[0,1]$ and all $j\in\{1,\ldots,n\}$ the \emph{generalized geodesic} $\mu_s$ connecting $\mu^2$ and $\mu^3$ (with \emph{base point} $\mu^1$) by
\begin{align*}
\mu_{s,j}:=\left[(1-s)\pi^2+s\pi^3\right]_\#\bm{\mu}_j,
\end{align*}
one has for all $s\in[0,1]$:
\begin{align}
\label{eq:gengeoconv}
\auxil(\mu_s)&\le (1-s)\auxil(\mu^2)+s\auxil(\mu^3)-\frac{\lambda}{2}s(1-s)\sum_{j=1}^n\frac1{m_j}\int_{\R^d\times\R^d\times\R^d}|x^3-x^2|^2\dd \bm{\mu}(x^1,x^2,x^3).
\end{align}
\end{itemize}
We call the energy \emph{uniformly} geodesically convex if it is $\lambda$-convex along generalized geodesics, for some $\lambda>0$. 
\end{definition}
Note that $\lambda$-convexity along generalized geodesics implies $\lambda$-convexity along geodesics in the usual sense, that is, for every pair $\mu^0,\mu^1\in\prbb$, there exists a constant-speed geodesic curve $(\mu^s)_{s\in [0,1]}$ in $\prbb$ connecting $\mu^0$ and $\mu^1$ for which
\begin{align}
\label{eq:geoconv}
\ent(\mu^s)&\le (1-s)\ent(\mu^0)+s\ent(\mu^1)-\frac12 s(1-s)\lambda \W_\prbb(\mu^0,\mu^1)^2,\qquad\forall~s\in [0,1].
\end{align}
For convexity along \emph{generalized} geodesics, an inequality of the form \eqref{eq:geoconv} is required for a wider class of curves joining $\mu^0$ and $\mu^1$.
The following sufficient criterion is useful in verifying convexity along generalized geodesics as it allows to consider \emph{absolutely continuous} measures and transport \emph{maps}.
\begin{thm}[Sufficient criterion for convexity along generalized geodesics {\cite[Prop.~9.2.10]{savare2008}}]\label{thm:suffgen}
Let \break $\auxil:\prbb\to\R_\infty$ be lower semicontinuous such that for all $\mu\in \prbb$, there exists a sequence $(\mu^k)_{k\in\N}$ on the subspace $\prbb^{\mathrm{ac}}$ of absolutely continuous measures in $\prbb$ with $\lim\limits_{k\to\infty}\W_\prbb(\mu^k,\mu)=0$ and $\lim\limits_{k\to\infty}\auxil(\mu^k)=\auxil(\mu)$. 

Assume moreover that for each $\mu\in\prbb^{\mathrm{ac}}$ and $t_1,\ldots,t_n,\tilde t_1,\ldots,\tilde t_n:\R^d\to\R^d$ such that $t_j-\tilde t_j\in L^2(\R;\dd \mu_j)$ for all $j\in \{1,\ldots,n\}$, the following estimate holds along the interpolating curve $(\mu^s)_{s\in[0,1]}$ defined as \break $\mu^s_j:=\left[(1-s)\tilde t_j+s t_j\right]{_\#}\mu_j~$ for all $s\in[0,1]$ and $j\in\{1,\ldots,n\}$:
\begin{align}\label{eq:suffgeoconv}
\auxil(\mu^s)\le (1-s)\auxil(\mu^0)+s\auxil(\mu^1)-\frac{\lambda}{2}s(1-s)\sum_{j=1}^n\frac1{m_j}\int_\Rd |t_j(x)-\tilde  t_j(x)|^2\dd \mu(x)\quad\text{for all }s\in[0,1].
\end{align}
Then, $\auxil$ is $\lambda$-convex along generalized geodesics on $(\prbb,\W_\prbb)$.
\end{thm}
\subsection{Geodesic convexity of the multi-component interaction energy}\label{subsec:cvx}
We first prove some basic properties of the interaction energy $\ent$.
\begin{lemma}[Proper domain and lower semicontinuity]\label{lem:basic}
The following statements hold:\vspace{-\topsep}
\begin{enumerate}[(a)]
\item For all $\mu\in\prbb$, one has $|\ent(\mu)|<\infty$.
\item $\ent$ is continuous on the metric space $(\prbb,\W_\prbb)$.
\item Let $\rho\in C^\infty_c(\Rd)$ be defined via 
\begin{align*}
\rho(x):=Z\exp\left(\frac1{|x|^2-1}\right)\eins{\ball_1(0)}(x),
\end{align*}
where $Z>0$ is such that $\|\rho\|_{L^1}=1$, and put $\rho_\eps(x):=\eps^{-d}\rho\left(\frac{x}{\eps}\right)$ for $\eps>0$. Then, for each $\mu\in\prbb$, the sequence $(\mu^k)_{k\in\N}$ with $\mu^k_j:=\rho_{\frac1{k}}\ast\mu_j$ for $k\in\N$, $j\in\{1,\ldots,n\}$, belongs to $\prbb^{\mathrm{ac}}$ and $\lim\limits_{k\to\infty}\W_\prbb(\mu^k,\mu)=0$.
\end{enumerate}
\end{lemma}
\begin{proof}
The key observation for parts (a) and (b) is the sub-quadratic growth of $\ent$ as a consequence of condition \textup{(W4)}: there exists a constant $C>0$ such that for all $\mu\in\prbb$, one has
\begin{align}\label{eq:entquadr}
|\ent(\mu)|\le C\sum_{j=1}^n\int_\Rd (1+|x|^2)\dd \mu_j(x).
\end{align}
Clearly, (a) follows. If $(\mu^k)_{k\in\N}$ is a sequence converging to $\mu$ in $(\prbb,\W_\prbb)$, in particular their second moments converge componentwise. Hence, the integrand on the r.h.s. in \eqref{eq:entquadr} is uniformly integrable which yields (b) as in \cite[Lemma 5.1.7]{savare2008}, using the continuity of $W$. For part (c), we observe for $j\in\{1,\ldots,n\}$ that $\mu^k_j$ is an absolutely continuous measure on $\R^d$ with Lebesgue density
\begin{align*}
x\mapsto \int_\Rd \rho_{\frac1{k}}(x-y)\dd\mu_j(y).
\end{align*}
Clearly, $\mu^k_j(\Rd)=p_j$. Moreover, the center of mass $E$ is unchanged by convolution with $\rho_{\frac1{k}}$ since for all $j\in\{1,\ldots,n\}$ and all $k\in\N$:
\begin{align*}
\int_\Rd x\dd\mu^k_j(x)&=\int_\Rd x\dd \mu_j(x).
\end{align*}
Indeed, by transformation and Fubini's theorem,
\begin{align*}
\int_\Rd x\dd\mu^k_j(x)&=\int_\Rd\int_\Rd (z+y)\rho_{\frac1{k}}(z)\dd\mu_j(y)\dd z=\int_\Rd y\dd\mu_j(y)+p_j\int_\Rd z\rho_{\frac1{k}}(z)\dd z.
\end{align*}
The last integral above vanishes since $\rho$ is an even function. Along the same lines, one proves convergence of the second moments:
\begin{align*}
\mom{\mu_j^k}&=\int_\Rd\int_\Rd |z+y|^2\rho_{\frac1{k}}(z)\dd\mu_j(y)\dd z\\&=\mom{\mu_j}+p_j\int_\Rd |z|^2\rho_{\frac1{k}}(z)\dd z+2\left(\int_\Rd z\rho_{\frac1{k}}(z)\dd z\right)^\tT\left(\int_\Rd y\dd \mu_j(y)\right).
\end{align*}
Since the last term vanishes again, we see
\begin{align*}
\mom{\mu_j^k}&=\mom{\mu_j}+\frac1{k^2}\int_\Rd |x|^2\rho(x)\dd x\stackrel{k\to\infty}{\longrightarrow}\mom{\mu_j}.
\end{align*}
It remains to prove narrow convergence of $\mu^k_j$ to $\mu_j$. Fix $f:\R^d\to\R$ continuous and bounded. Using Fubini's theorem again, we get, since $\rho$ is even,
\begin{align*}
\int_\Rd f\dd \mu_j^k-\int_\Rd f\dd\mu_j&=\int_\Rd (\rho_{\frac1{k}}\ast f-f)\dd\mu_j.
\end{align*}
Since $f$ is continuous, $\rho_{\frac1{k}}\ast f$ converges to $f$ pointwise on $\R^d$ (see, for instance, \cite[App. C]{evans2010}). Clearly, $\rho_{\frac1{k}}\ast f$ is $k$-uniformly bounded. The dominated convergence theorem now yields
\begin{align*}
\int_\Rd (\rho_{\frac1{k}}\ast f-f)\dd\mu_j\stackrel{k\to\infty}{\longrightarrow}0,
\end{align*}
proving the claim.
\end{proof}
\begin{lemma}[Growth control on the gradient]\label{lem:grad}
There exists a matrix $\overline C\in\Matn$ such that for all $z\in\Rd$ and all $i,j\in\{1,\ldots,n\}$:
\begin{align}
\label{eq:grad}
|\nabla W_{ij}(z)|&\le \overline{C}_{ij}(|z|+1).
\end{align}
\end{lemma}
\begin{proof}
We give a short proof for the sake of completeness. From (W2) and (W5), it easily follows for all $x,y\in\Rd$ that
\begin{align*}
W_{ij}(y)-W_{ij}(x)-\frac{\kappa_{ij}}{2}|y-x|^2&\ge \nabla W_{ij}(x)^\tT(y-x).
\end{align*}
Putting 
\begin{align*}
\alpha:=\begin{cases}|4\bar W_{ij}-\kappa_{ij}|^{-1}&\text{if }4\bar W_{ij}>\kappa_{ij},\\ 1&\text{otherwise,}\end{cases}\quad\text{and}\quad y:=x+\alpha \nabla W_{ij}(x),
\end{align*}
we get, using (W4) and Young's inequality:
\begin{align*}
\alpha|\nabla W_{ij}(x)|^2&\le \overline{W}_{ij}(2+3|x|^2)+\frac12 \alpha^2 (4\bar W_{ij}-\kappa_{ij})|\nabla W(x)|^2.
\end{align*}
Consequently, in both cases, we have
\begin{align*}
|\nabla W_{ij}(x)|&\le \left(2\alpha^{-1}\overline{W}_{ij}(2+3|x|^2)\right)^{1/2}
\end{align*}
which implies an estimate of the form \eqref{eq:grad} via the elementary estimate $\sqrt{a+b}\le\sqrt{a}+\sqrt{b}$ holding for $a,b\ge 0$.
\end{proof}
\begin{remark}[Invariants]\label{rem:inva}
Along the flow of system \eqref{eq:pdend}, the set $\prbb$ is positively invariant. We give a formal indication of this fact: Let an initial datum $\mu^0\in\prbb$ be given. Since \eqref{eq:pdend} is in divergence form, we immediately obtain the conservation of mass:
\begin{align*}
\fder{t}\int_\Rd \dd \mu_i(t,x)&=0.
\end{align*}
Furthermore, by formal integration by parts, one has
\begin{align*}
\fder{t}\sum_{i=1}^n\mom{\mu_i(t)}&=-\sum_{i=1}^n 2m_i\int_\Rd x^\tT\left(\sum_{j=1}^n \nabla W_{ij}\ast\mu_j(t)\right)(x)\dd \mu_i(t,x),
\end{align*}
from which it is possible to derive using the Young and Jensen inequalities and Lemma \ref{lem:grad} the estimate
\begin{align*}
\fder{t}\sum_{i=1}^n\mom{\mu_i(t)}&\le A\sum_{i=1}^n\mom{\mu_i(t)}+B,
\end{align*}
for suitable $A,B\in\R$. Gronwall's lemma now yields finiteness of second moments at a fixed time $t\ge 0$. Finally,
\begin{align*}
\fder{t}\sum_{i=1}^n \frac1{m_i}\int_\Rd x\dd \mu_i(t,x)=-\sum_{i=1}^n\sum_{j=1}^n\int_\Rd\int_\Rd\nabla W_{ij}(x-y)\dd \mu_j(t,y)\dd\mu_i(t,x).
\end{align*}
Using assumptions \textup{(W1)} and \textup{(W3)} in combination with Fubini's theorem, we observe that the r.h.s. above is in fact equal to 0.
\end{remark}
\begin{definition}[Irreducible systems]\label{def:irred}
We call a system of the form \eqref{eq:pdend} \emph{irreducible}, if the graph \break $G=(V_G,E_G)$ with nodes $V_G=\{1,\ldots,n\}$ and edges $E_G=\{(i,j)\in V_G\times V_G:\,\nabla W_{ij}\not\equiv 0 \text{ on }\Rd\}$ is connected. That is, irreducible systems cannot be split up into independent subsystems.
\end{definition}

The main result of this section is concerned with the geodesic convexity of the interaction energy $\ent$:
\begin{thm}[Criterion for geodesic convexity]\label{thm:geoconv}
Let $n>1$ and let \eqref{eq:pdend} be irreducible. Define for \break $i\in \{1,\ldots,n\}$ the quantity $\eta_i:=\min\limits_{j\neq i}\kappa_{ij}m_j\in\R$. Then, $\ent$ is $\lambda$-convex along generalized geodesics on $\prbb$ w.r.t. $\W_\prbb$ for all $\lambda\le \lambda_0$ with
\begin{align}
\label{eq:convcond0}
\lambda_0&:=\min_{i\in\{1,\ldots,n\}}\left[p_i\min(0,m_i\kappa_{ii}-\eta_i)+\frac12\sum_{j=1}^n p_j\left(\eta_j+\eta_i\frac{m_i}{m_j}\right)\right].
\end{align}
\end{thm}
\begin{proof}
Thanks to the properties from Lemma \ref{lem:basic}, we are allowed to use Theorem \ref{thm:suffgen}. Let therefore $\mu\in \prbb^{\mathrm{ac}}$ and let $t_1,\ldots,t_n,\tilde t_1,\ldots,\tilde t_n:\R^d\to\R^d$ such that $t_j-\tilde t_j\in L^2(\R;\dd \mu_j)$ for all $j\in \{1,\ldots,n\}$. With the notation from Theorem \ref{thm:suffgen}, we have, using the convexity condition (W5) and the definition of the push-forward:
\begin{align*}
&\ent(\mu^s)=\frac12 \sum_{i=1}^n\sum_{j=1}^n \int_\Rd\int_\Rd W_{ij}(\tilde t_j(x)-\tilde t_i(y)+s[t_j(x)-t_i(y)-(\tilde t_j(x)-\tilde t_i(y))])\dd \mu_j(x)\dd\mu_i(y)\\
&\le (1-s)\ent(\mu^0)+s\ent(\mu^1)-\frac12 s(1-s)\cdot\frac12 \sum_{i=1}^n\sum_{j=1}^n \int_\Rd\int_\Rd  \kappa_{ij}|t_j(x)-t_i(y)-(\tilde t_j(x)-\tilde t_i(y))|^2\dd \mu_j(x)\dd\mu_i(y).
\end{align*}
In view of \eqref{eq:gengeoconv}, we have to verify that
\begin{align}
\label{eq:convproof}
\begin{split}
&\frac12 \sum_{i=1}^n\sum_{j=1}^n \int_\Rd\int_\Rd  \kappa_{ij}|t_j(x)-t_i(y)-(\tilde t_j(x)-\tilde t_i(y))|^2\dd \mu_j(x)\dd\mu_i(y)\\\quad&\ge \lambda_0 \sum_{i=1}^n \frac1{m_i}\int_\Rd |t_i(x)-\tilde t_i(x)|^2\dd\mu_i(x).
\end{split}
\end{align}
We first split up the l.h.s. of \eqref{eq:convproof} into its diagonal and off-diagonal part and perform an estimate on the latter introducing the numbers $\eta_i=\min\limits_{j\neq i}\kappa_{ij}m_j$:
\begin{align*}
\frac12 \sum_{i=1}^n\sum_{j=1}^n &\int_\Rd\int_\Rd  \kappa_{ij}|t_j(x)-t_i(y)-(\tilde t_j(x)-\tilde t_i(y))|^2\dd \mu_j(x)\dd\mu_i(y)\\
&\ge \frac12 \sum_i\sum_{j\neq i} \int_\Rd\int_\Rd \frac{\eta_i}{m_j}|(t_j(x)-\tilde t_j(x))-(t_i(y)-\tilde t_i(y))|^2 \dd\mu_j(x)\dd\mu_i(y)\\
&\quad+\frac12 \sum_i \int_\Rd\int_\Rd \kappa_{ii}|(t_i(x)-\tilde t_i(x))-(t_i(y)-\tilde t_i(y))|^2 \dd\mu_i(x)\dd\mu_i(y).
\end{align*}
Expanding the squares yields 
\begin{align}
&\frac12 \sum_i\sum_{j\neq i} \int_\Rd\int_\Rd \frac{\eta_i}{m_j}|(t_j(x)-\tilde t_j(x))-(t_i(y)-\tilde t_i(y))|^2 \dd\mu_j(x)\dd\mu_i(y)\nonumber\\
&\quad+\frac12 \sum_i \int_\Rd\int_\Rd \kappa_{ii}|(t_i(x)-\tilde t_i(x))-(t_i(y)-\tilde t_i(y))|^2 \dd\mu_i(x)\dd\mu_i(y)\nonumber\\
&=\frac12 \sum_i\sum_{j\neq i} \left(\int_\Rd\frac{p_i\eta_i}{m_j}|t_j(x)-\tilde t_j(x)|^2\dd \mu_j(x)+\int_\Rd\frac{p_j\eta_i}{m_j}|t_i(x)-\tilde t_i(x)|^2\dd \mu_i(x)\right)\label{eq:expandedsquares}\\
&\quad-\sum_i\left(\sum_{j\neq i}\int_\Rd\frac1{m_j}(t_j(x)-\tilde t_j(x))\dd\mu_j(x)\right)^\tT\left(\int_\Rd \eta_i(t_i(x)-\tilde t_i(x))\dd \mu_i(x)\right)\nonumber\\
&\quad+\sum_i\kappa_{ii}\left(\int_\Rd p_i|t_i(x)-\tilde  t_i(x)|^2\dd\mu_i(x)\right)-\sum_i\kappa_{ii}\left|\int_\Rd (t_i(x)-\tilde t_i(x))\dd \mu_i(x)\right|^2.\nonumber
\end{align}
Now, the special structure of $\prbb$ comes into play: since the weighted center of mass $E$ is fixed on $\prbb$, one has
\begin{align*}
E&=\sum_{j=1}^n\frac1{m_j}\int_\Rd x\dd(t_j{_\#}\mu_j)=\sum_{j=1}^n\frac1{m_j}\int_\Rd x\dd(\tilde t_j{_\#}\mu_j),
\end{align*}
and consequently
\begin{align*}
\sum_{j\neq i}\int_\Rd\frac1{m_j}(t_j(x)-\tilde t_j(x))\dd\mu_j(x)=-\int_\Rd\frac1{m_i}(t_i(x)-\tilde t_i(x))\dd\mu_i(x).
\end{align*}
We exploit this fact in order to simplify the second term on the r.h.s. of formula \eqref{eq:expandedsquares} above:
\begin{align*}
\mathrm{r.h.s.}=\sum_i\bigg\{&\left|\int_\Rd (t_i(x)-\tilde t_i(x))\dd\left(\frac1{p_i}{\mu_i}\right)(x)\right|^2 p_i^2\left(\frac{\eta_i}{m_i}-\kappa_{ii}\right)\bigg.\\&+\int_\Rd |t_i(x)-\tilde t_i(x)|^2\dd\left(\frac1{p_i}{\mu_i}\right)(x)p_i^2\left(\kappa_{ii}-\frac{\eta_i}{m_i}\right)\\
&+\bigg.\frac1{m_i}\int_\Rd |t_i(x)-\tilde t_i(x)|^2\dd\mu_i(x)\cdot \frac12 \sum_j p_j\left(\eta_j+\eta_i\frac{m_i}{m_j}\right)\bigg\}=:\sum_{i}S_i.
\end{align*}
We analyse each $S_i$ separately.\\
If $\frac{\eta_i}{m_i}-\kappa_{ii}\ge 0$, the first term in $S_i$ is nonnegative, so
\begin{align*}
S_i&\ge \frac1{m_i}\int_\Rd |t_i(x)-\tilde t_i(x)|^2\dd\mu_i(x)\cdot\left[p_i(m_i\kappa_{ii}-\eta_i)+\frac12\sum_jp_j\left(\eta_j+\eta_i\frac{m_i}{m_j}\right)\right].
\end{align*}
If $\frac{\eta_i}{m_i}-\kappa_{ii}< 0$, the sum of the first two terms in $S_i$ is nonnegative thanks to Jensen's inequality. Hence,
\begin{align*}
S_i&\ge \frac1{m_i}\int_\Rd |t_i(x)-\tilde t_i(x)|^2\dd\mu_i(x)\cdot\frac12\sum_jp_j\left(\eta_j+\eta_i\frac{m_i}{m_j}\right).
\end{align*}
Defining $\lambda_0$ as in \eqref{eq:convcond0} clearly leads to \eqref{eq:convproof}, completing the proof.
\end{proof}
\begin{remark}[Non-irreducible systems]
If system \eqref{eq:pdend} is \emph{not} irreducible, there exists an $I$-integer partition ($I\in\N$) of $n\in\N$ into $n_1+n_2+\ldots+n_I=n$ such that \eqref{eq:pdend} decomposes into $I$ independent irreducible subsystems having the same structure as \eqref{eq:pdend}, but with $n$ replaced by $n_1,\ldots,n_I$, respectively. The modulus of geodesic convexity of the interaction energy $\ent$ can now be computed as the minimum of the respective convexity moduli of each subsystem: if $n_k>1$ for some $k\in\{1,\ldots,I\}$, formula \eqref{eq:convcond0} applies; if $n_k=1$, McCann's criterion \cite{mccann1997} applies (and yields convexity modulus $m\kappa p$ for the respective $m,\kappa,p$ of the $k^{\mathrm{th}}$ subsystem in our framework).
\end{remark}
\begin{prop}[Necessary condition for $\lambda_0>0$]\label{rem:necconv}
If $\lambda_0>0$ in \eqref{eq:convcond0}, then for all $i\in\{1,\ldots,n\}$:
\begin{align*}
\sum_{j=1}^n \kappa_{ij}p_j&>0.
\end{align*}
This condition is \emph{not} sufficient (see the examples from the introduction).
\end{prop}
\begin{proof}
Fix $i\in\{1,\ldots,n\}$. The following holds:
\begin{align*}
m_i\sum_j \kappa_{ij}p_j&=p_im_i\kappa_{ii}+\sum_{j\neq i}m_i\kappa_{ij}p_j=p_i(m_i\kappa_{ii}-\eta_i)+p_i\eta_i+\frac12 \sum_{j\neq i}p_j\kappa_{ji}m_i+\frac12\sum_{j\neq i}p_j\kappa_{ij}m_j\frac{m_i}{m_j},
\end{align*}
where we used the symmetry of $\kappa$. Now, we estimate using the definition of $\eta_i,\eta_j$:
\begin{align*}
m_i\sum_j \kappa_{ij}p_j&\ge p_i\min(0,m_i\kappa_{ii}-\eta_i)+\frac12 p_i\eta_i+\frac12 \sum_{j\neq i}p_j\eta_j+\frac12 p_i\eta_i\frac{m_i}{m_i}+\frac12 \sum_{j\neq i}p_j\eta_i\frac{m_i}{m_j}\\
&=p_i\min(0,m_i\kappa_{ii}-\eta_i)+\frac12\sum_j p_j\left(\eta_j+\eta_i\frac{m_i}{m_j}\right)\ge \lambda_0>0.
\end{align*}
\end{proof}

\subsection{Existence and uniqueness of gradient flow solutions}\label{subsec:gfsol}
With the results of Lemma \ref{lem:basic} and Theorem \ref{thm:geoconv} at hand, the following statement follows thanks to \cite[Chapter 11]{savare2008}:
\begin{thm}[Existence and uniqueness]\label{thm:exun}
Consider \eqref{eq:pdend} endowed with an inital datum $\mu^0\in\prb$. Then, there exists a \emph{gradient flow} solution $\mu\in AC^2_{\mathrm{loc}}([0,\infty);(\mathscr{P},\W_\prbb))$ to this initial-value problem: System \eqref{eq:pdend} holds in the sense of distributions and one has $\mu(0)=\mu^0$. Moreover, with $\lambda_0$ from \eqref{eq:convcond0}, the \emph{evolution variational estimate} holds for almost every $t>0$ and all $\nu\in\prb$:
\begin{align*}
\frac12 \frac{\dd^+}{\dd t}\W_\prbb(\mu(t),\nu)^2+\frac{\lambda_0}{2}\W_\prbb(\mu(t),\nu)^2\le \ent(\nu)-\ent(\mu(t)).
\end{align*}
Given another initial datum $\nu^0\in \mathscr{P}$ and the respective gradient flow solution $\nu\in AC^2_{\mathrm{loc}}([0,\infty);(\mathscr{P},\W_\prbb))$, the following \emph{contraction estimate} holds for all $t\ge 0$:
\begin{align}
\label{eq:contr}
\W_\prbb(\mu(t),\nu(t))&\le e^{-\lambda_0 t}\W_\prbb(\mu^0,\nu^0),
\end{align}
which implies in particular the uniqueness of solutions.
\end{thm}
\begin{coro}[The uniformly convex case]\label{coro:posconv}
If \eqref{eq:convcond0} yields $\lambda_0>0$, the measure
\begin{align*}
\mu^\infty:=(p_1,\ldots,p_n)^\tT\delta_{x^\infty},\qquad\text{with}\qquad x^\infty:=E\left[\sum_{j=1}^n\frac{p_j}{m_j}\right]^{-1}\in\Rd,
\end{align*}
is the unique minimizer of $\ent$ and the unique stationary state of \eqref{eq:pdend} on $\mathscr{P}$. It is globally asymptotically stable: The solution from Theorem \ref{thm:exun} converges exponentially fast in $(\prb,\W_\prbb)$ at rate $\lambda_0$ to $\mu^\infty$.
\end{coro}
\begin{proof}
One easily sees---thanks to the symmetry property $\nabla W_{ij}(0)=0$ from conditions (W2)\&(W3)---that $\mu^\infty$ is a steady state of \eqref{eq:pdend}. It is unique by uniform geodesic convexity of $\ent$: the contraction estimate \eqref{eq:contr} necessarily implies that two steady states on $\mathscr{P}$ coincide. On the other hand, uniform geodesic convexity of $\ent$ implies existence and uniqueness of a minimizer of $\ent$ which is---since $\ent$ is nonincreasing along the solution $\mu(t)$---a steady state of \eqref{eq:pdend}. 
\end{proof}
As for scalar equations of the form \eqref{eq:pde1d}, system \eqref{eq:pdend} can be viewed as a continuum limit of a multi-particle system. To this end, we introduce the concept of \emph{particle solutions} as a conclusion to this section.
\begin{remark}[Particle solutions]
Assume that the initial datum is discrete, i.e. each component $\mu^0_i$ is a finite linear combination of Dirac measures:
\begin{align*}
\mu_i^0=\sum_{k=1}^{N_i}p_i^{k}\delta_{x_{i}^{0,k}} \quad(i=1,\ldots,n).
\end{align*}
There, the $N_i\in\N$ \emph{particles} of species $i$ have \emph{mass} $p_i^k>0$ and are at initial \emph{position} $x_{i}^{0,k}\in\R^d$, for $k=1,\ldots,N_i$, respectively. Let $N:=\sum\limits_{i=1}^n N_i$ and let a family $\rx=(x_i^k)$ ($k=1,\ldots,N_i$; $i=1,\ldots,n$) of $L^2$-absolutely continuous curves $x_i^k\in AC^2([0,\infty);\R^d)$ be given, such that the following initial-value problem for a system of $N$ ordinary differential equations on $\Rd$ is globally solved:
\begin{align}
\label{eq:particle}
\begin{split}
\fder{t}x_i^k(t)&=-m_i\sum_{j=1}^n\sum_{l=1}^{N_j}p_j^l\nabla W_{ij}(x_i^k(t)-x_j^l(t)),\qquad x_i^k(0)=x_{i}^{0,k},\qquad (k=1,\ldots,N_i;~i=1,\ldots,n).
\end{split}
\end{align}
Then it is easy to verify that the \emph{particle solution}
\begin{align}
\label{eq:atomic}
\mu_i(t)=\sum_{k=1}^{N_i}p_i^{k}\delta_{x_{i}^k(t)} \quad(i=1,\ldots,n)
\end{align}
is the unique gradient flow solution to system \eqref{eq:pdend} with initial datum $\mu^0$ given above. However, it is a non-trivial question if such $\rx$ exist, since (W1)-(W5) do not imply global Lipschitz-continuity of the r.h.s. in \eqref{eq:particle}. Nevertheless, \eqref{eq:particle} admits locally absolutely continuous solutions since this system possesses an underlying (discrete) gradient flow structure: define the finite-dimensional space
\begin{align*}
\prb_\mathrm{d}:=\left\{\rx\in \prod_{i=1}^n\prod_{k=1}^{N_i}\Rd\cong \R^{Nd}:\,p_i=\sum_{k=1}^{N_i}p_i^k~(i=1,\ldots,n);\,E=\sum_{i=1}^n\frac1{m_i}\sum_{k=1}^{N_i}p_i^kx_i^k\right\},
\end{align*}
endowed with the (weighted Euclidean) distance
\begin{align*}
\metr(\rx,\ry)=\left[\sum_{i=1}^n\frac1{m_i}\sum_{k=1}^{N_i}p_i^k|x_i^k-y_i^k|^2\right]^{1/2},
\end{align*}
and define the \emph{discrete interaction energy} $\ent_\mathrm{d}$ on $\prb_\mathrm{d}$ as
\begin{align*}
\ent_\mathrm{d}(\rx):=\frac12 \sum_{i=1}^n\sum_{j=1}^n\sum_{k=1}^{N_i}\sum_{l=1}^{N_j}p_i^k p_j^l W_{ij}(x_i^k-x_j^l).
\end{align*}
Applying the same method of proof as for Theorem \ref{thm:geoconv} \emph{mutatis mutandis} for the discrete framework, one can show that $\ent_{\mathrm{d}}$ is $\lambda_0$-geodesically convex on $(\prb_{\mathrm{d}},\metr)$ with \emph{the same} modulus of convexity $\lambda_0$ as in the continuous case \eqref{eq:convcond0}. We can again invoke \cite{savare2008} to obtain existence and uniqueness of a solution curve $\rx\in AC^2_{\mathrm{loc}}([0,\infty);(\prb_{\mathrm{d}},\metr))$ to the particle system \eqref{eq:particle}. Conversely, thanks to the uniqueness of solutions to both \eqref{eq:pdend} and \eqref{eq:particle}, a gradient flow solution $\mu$ to \eqref{eq:pdend} of the form \eqref{eq:atomic} can be represented by a solution $\rx$ to \eqref{eq:particle}.
\end{remark}


\section{Qualitative properties in one spatial dimension}\label{sec:qual}
In this section, we analyse the qualitative behaviour of the solution from Theorem \ref{thm:exun} in the general scenario, i.e. the criterion for geodesic convexity may only yield $\lambda_0\le 0$. In this case, the contraction estimate \eqref{eq:contr} does not allow for conclusions on the long-time behaviour of the solution.

From now on, consider \eqref{eq:pdend} in one spatial dimension $d=1$; and let $\mu$ be the solution to \eqref{eq:pdend} with initial datum $\mu^0\in\prb$, as given in Theorem \ref{thm:exun}. First, we rewrite system \eqref{eq:pdend} in terms of the inverse distribution functions $u=(u_1,\ldots,u_n)$; recall their definition from \eqref{eq:cumF}\&\eqref{eq:invu}.

For all $z\in [0,1)$, one has $z=F_i(t,u_i(t,z))$. Differentiation w.r.t. $t$ yields
\begin{align*}
0&=\partial_t F_i(t,u_i(t,z))+\partial_x F_i(t,u_i(t,z))\partial_t u_i(t,z)\\
&=\int_{-\infty}^{u_i(t,z)}\frac1{p_i}\partial_y\left(\sum_{j=1}^nm_i\mu_i(t,y)(W_{ij}'\ast\mu_j)(t,y)\right)\dd y+\frac1{p_i}\mu_i(t,u_i(t,z))\partial_t u_i(t,z)\\
&=\frac{m_i}{p_i}\sum_{j=1}^n\mu_i(t,u_i(t,z))\int_\R W_{ij}'(u_i(t,z)-y)\mu_j(t,y)\dd y+\frac1{p_i}\mu_i(t,u_i(t,z))\partial_t u_i(t,z).
\end{align*}
Rearranging yields -- with the help of (W3) and the transformation $\xi:=F_j(t,y)$:
\begin{align}
\label{eq:u}
\partial_t u_i(t,z)&=m_i\sum_{j=1}^n p_j\int_0^1 W_{ij}'(u_j(t,\xi)-u_i(t,z))\dd \xi\qquad (i=1,\ldots,n).
\end{align}
It is a consequence of Theorem \ref{thm:exun} that given a gradient flow solution $\mu$ to \eqref{eq:pdend}, the corresponding curve of pseudo-inverse distribution functions $u\in AC^2_{\mathrm{loc}}([0,\infty);L^2([0,1];\R^n))$ solves \eqref{eq:u}. Furthermore, since $\mu(t)\in\prb$ for all $t\ge 0$, $u_i(t,\cdot)$ is a non-decreasing \emph{càdlàg} function on $(0,1)$. Conservation of the weighted center of mass $E$ over time is reflected in terms of $u$ by the identity
\begin{align}
\label{eq:uE}
E&=\sum_{j=1}^n \frac{p_j}{m_j}\int_0^1 u_j(t,z)\dd z\qquad\forall t\ge 0.
\end{align}
The concept of inverse distribution functions substantially simplifies the analysis of solutions to \eqref{eq:pdend} since there does not appear any spatial derivative on the right-hand side of \eqref{eq:u} anymore. However, this approach can be employed in one spatial dimension $d=1$ only.

\subsection{Speed of propagation and confinement}\label{subsec:speed}
In this section, we investigate the rate of propagation of the solution to \eqref{eq:pdend} in space over time, given an initial datum with compact support. We first obtain -- for arbitrary potentials satisfying (W1)-(W5) -- boundedness of the support of $\mu(t)$ for fixed time $t>0$, and second -- under more restrictive requirements on the potential $W$ -- $t$-\emph{uniform} boundedness of $\supp\mu(t)$.
\begin{prop}[Finite speed of propagation]
\label{prop:finsupp}
Let an initial datum $\mu^0$ with compact support and $T>0$ be given. Then, there exists a constant $K=K(T,\mu^0)>0$ such that for all $t\in [0,T]$,
\begin{align*}
\supp \mu(t)&\subset [-K,K].
\end{align*}
\end{prop}
\begin{proof}
For $t\in [0,T]$ and $i\in\{1,\ldots,n\}$, denote
\begin{align*}
u_i(t,1^-):=\lim_{\eps\searrow 0}u_i(t,1-\eps)\in \R\cup\{+\infty\}\quad\text{and}\quad u_i(t,0^+):=\lim_{\eps\searrow 0}u_i(t,\eps)\in\R\cup \{-\infty\}.
\end{align*}
The assertion will follow from finiteness of those limits. Let $\eps>0$. Then,
\begin{align*}
\partial_t (u_i(t,1-\eps)^2)&\le 2|u_i(t,1-\eps)|\sum_{j=1}^nm_ip_j\int_0^1|W_{ij}'(u_j(t,\xi)-u_i(t,1-\eps))|\dd\xi.
\end{align*}
Lemma \ref{lem:grad}, H\"older's and Young's inequality eventually lead to
\begin{align}
\label{eq:fsproof1}
\begin{split}
&2|u_i(t,1-\eps)|\sum_{j=1}^nm_ip_j\int_0^1|W_{ij}'(u_j(t,\xi)-u_i(t,1-\eps))|\dd\xi\\
&\le 2|u_i(t,1-\eps)|\sum_{j=1}^n\overline{C}_{ij}m_ip_j\left(\int_0^1 |u_j(t,\xi)|\dd\xi+|u_i(t,1-\eps)|+1\right)\\
&\le 2\left[\sum_{j=1}^n \overline{C}_{ij}m_ip_j+1\right]u_i(t,1-\eps)^2+\left(\sum_{j=1}^n \overline{C}_{ij}m_ip_j\right)^2+2\max_j\left(\overline{C}_{ij}^2m_i^2m_jp_j\right)\sum_{j=1}^n\int_0^1\frac{p_j}{m_j}u_j(t,\xi)^2\dd\xi.
\end{split}
\end{align}
With the transformation $\xi:=F_j(t,x)$, we observe that the sum in the last term on the r.h.s. of \eqref{eq:fsproof1} can be expressed in terms of the second moments $\mom{\mu_j(t)}$ and of $\W_\prbb(\mu(t),\delta_0\einsvec)$, where $\einsvec=(1,1,\ldots,1)^\tT\in\R^n$:
\begin{align*}
\sum_{j=1}^n\int_0^1 \frac{p_j}{m_j}u_j(t,\xi)^2\dd\xi&=\sum_{j=1}^n\frac1{m_j}\mom{\mu_j(t)}=\W_\prbb^2(\mu(t),\delta_0\einsvec).
\end{align*}
Since $\mu\in AC^2([0,T];(\prb,\W_\prbb))$, there exists $\varphi\in L^2([0,T])$ such that
\begin{align*}
\W_\prbb(\mu(t),\mu^0)&\le \int_0^t \varphi(s)\dd s,\qquad\forall~t\in[0,T].
\end{align*}
We obtain
\begin{align*}
\sum_{j=1}^n\int_0^1 \frac{p_j}{m_j}u_j(t,\xi)^2\dd\xi&\le 2\W_\prbb^2(\mu(t),\mu^0)+2\W_\prbb^2(\mu^0,\delta_0\einsvec)\le 2\left(\int_0^t \varphi(s)\dd s\right)^2+2\W_\prbb^2(\mu^0,\delta_0\einsvec)\\
&\le 2T\|\varphi\|_{L^2([0,T])}^2+2\sum_{j=1}^n\frac1{m_j}\mom{\mu^0_j},
\end{align*}
which is a constant depending on $T$ and $\mu^0$. Inserting into \eqref{eq:fsproof1}, we observe
\begin{align*}
\partial_t(u_i(t,1-\eps)^2)&\le Au_i(t,1-\eps)^2+B(T,\mu^0),
\end{align*}
for suitable constants $A,B>0$. We apply Gronwall's lemma, let $\eps\searrow 0$ and use that -- since $\mu^0$ has compact support by assumption -- the limit $u_i(0,1^-)$ exists in $\R$:
\begin{align*}
u_i(t,1^-)^2&\le \left[u_i(0,1^-)^2+\frac{B}{A}\right]\exp(AT)\quad\forall~t\in[0,T].
\end{align*}
Thus, $u_i(t,1^-)$ is a finite value, at each $t\in[0,T]$. Along the same lines, it can be shown that
\begin{align*}
u_i(t,0^+)^2&\le \left[u_i(0,0^+)^2+\frac{B}{A}\right]\exp(AT)\quad\forall~t\in[0,T].
\end{align*}
Since by construction $\supp\mu_i(t)\subset [u_i(t,0^+),u_i(t,1^-)]$, the assertion is proven.
\end{proof}
The statement of Proposition \ref{prop:finsupp} shows that at fixed $t\ge 0$, the limits
\begin{align*}
u_i(t,1^-):=\lim_{\eps\searrow 0}u_i(t,1-\eps)\quad\text{and}\quad u_i(t,0^+):=\lim_{\eps\searrow 0}u_i(t,\eps)
\end{align*}
exist (in $\R$). In order to prove uniform confinement of the solution, we show uniform boundedness of those limits. We first introduce a requirement on the potential by the following
\begin{definition}[Confining potentials]\label{def:selfc}
We call an interaction potential $W$ satisfying \textup{(W1)--(W5)} \emph{confining} if there exists $R>0$ such that:
\begin{enumerate}[(i)]
\item System \eqref{eq:pdend} is \emph{irreducible at large distance}, that is the graph $G'=(V_{G'},E_{G'})$ with nodes $V_{G'}=\{1,\ldots,n\}$ and edges $E_{G'}=\{(i,j)\in V_{G'}\times V_{G'}:\,W_{ij}'\not\equiv 0 \text{ on }(R,\infty)\}$ is connected.
\item There exists a matrix $C\in\Matn$ such that for each $i,j\in\{1,\ldots,n\}$, the map $W_{ij}$ is $C_{ij}$-(semi-)convex on the interval $(R,\infty)$ and the following holds: \\ 
If $n=1$, then $C>0$. If $n>1$, with $\tilde\eta_i:=\min\limits_{j\neq i}C_{ij}p_j$ for all $i\in\{1,\ldots,n\}$,
\begin{align}
\label{eq:lself}
\tilde\lambda_0&:=\min_{i\in\{1,\ldots,n\}}\left[p_i\min(0,m_iC_{ii}-\tilde\eta_i)+\frac12\sum_{j=1}^n p_j\left(\tilde\eta_j+\tilde\eta_i\frac{m_i}{m_j}\right)\right]>0.
\end{align}
\end{enumerate}
\end{definition}
\begin{remark}[Geodesic convexity and confinement]
In the scalar case $n=1$, uniform geodesic convexity of the interaction energy $\ent$ is \emph{equivalent} to $\kappa$-convexity of $W$ with $\kappa>0$ \cite{mccann1997}. So, the potential is confining. Also for genuine systems, if $\lambda_0>0$ in \eqref{eq:convcond0}, the definition $C:=\kappa$ yields $\tilde\lambda_0=\lambda_0>0$. Hence, our criterion for uniform geodesic convexity of $\ent$ necessarily implies that $W$ is a confining potential. Naturally, if the system is not irreducible at large distance, the independent irreducible subsystems should be considered separately.
\end{remark}

\begin{thm}[Confinement]
\label{prop:confsupp}
Assume that $W$ is confining and let $\mu^0$ have compact support. Then, there exists a constant $K=K(\mu^0)>0$ independent of $t$ such that for all $t\ge 0$:
\begin{align}
\label{eq:confsupp}
\supp \mu(t)&\subset [-K,K].
\end{align}
\end{thm}
\begin{proof}
We prove the assertion in the case of genuine systems $n>1$.\\
\noindent \underline{Step 1:} estimate on second moments of $\mu$.\\
Let $\eps>0$ be sufficiently small such that replacing $C_{ij}$ by $C_{ij}^\eps:=C_{ij}-\eps$ in \eqref{eq:lself} still yields a number $\tilde\lambda_0^\eps>0$, possibly with $\tilde\lambda_0^\eps <\tilde\lambda_0$. From the $C_{ij}$-convexity of $W_{ij}$  on $(R,\infty)$ and with the help of Young's inequality, we get for all $z>R$:
\begin{align*}
W_{ij}(z)&\ge W_{ij}(R)+W_{ij}'(R)(z-R)+\frac12 C_{ij}(z-R)^2\ge \frac12 C_{ij}^\eps z^2-D_{ij},
\end{align*}
for appropriate constants $D_{ij}>0$. Thanks to (W2)\&(W3), enlarging the constants, there exists $D>0$ such that for all $i,j\in\{1,\ldots,n\}$ and \emph{for all} $z\in\R$:
\begin{align}
\label{eq:Wper}
W_{ij}(z)&\ge \frac12 C_{ij}^\eps z^2-D.
\end{align}
We now use boundedness of the energy $\ent$ along the (gradient flow) solution to obtain with \eqref{eq:Wper} for all $t\ge 0$:
\begin{align*}
2\ent(\mu^0)\ge 2 \ent(\mu(t))&\ge \frac12\sum_{i=1}^n\sum_{j=1}^n\int_\R\int_\R C_{ij}^\eps(x-y)^2\dd \mu_j(x)\dd\mu_i(y)-D\left(\sum_{j=1}^n p_j\right)^2.
\end{align*}
The first term on the r.h.s. has precisely the same structure as the l.h.s. in \eqref{eq:convproof} for $t_i=t_j\equiv 0$ and $\tilde t_i=\tilde t_j=\id$. Arguing exactly as in the proof of Theorem \ref{thm:geoconv}, we obtain
\begin{align*}
2\ent(\mu^0)&\ge \tilde\lambda_0^\eps \sum_{j=1}^n\frac1{m_j}\mom{\mu_j(t)}-D\left(\sum_{j=1}^n p_j\right)^2.
\end{align*}
All in all, we have proven uniform boundedness of the second moments: There exists $C_2>0$ auch that for all $t\ge 0$ and all $i\in\{1,\ldots,n\}$, one has $\mom{\mu_i(t)}\le C_2$.

\noindent \underline{Step 2:} $L^\infty$ estimate for $u$.\\
We first prove an upper bound. For each $t\ge 0$, we consider those indices $i\in\{1,\ldots,n\}$, where
\begin{align*}
u_i(t,1^-)&\ge \max_{j\in\{1,\ldots,n\}}u_j(t,1^-)-R.
\end{align*}
That is, for all $\xi\in [0,1)$ and all $j\in\{1,\ldots,n\}$: $u_i(t,1^-)\ge u_j(t,\xi)-R$. We thus have, for each $j\in\{1,\ldots,n\}$, a partition of $[0,1)$ into two sets $A_1^j$ and $A_2^j$, where
\begin{align*}
A_1^j:=\{\xi\in[0,1):\,u_j(t,\xi)-u_i(t,1^-)<-R\},\qquad A_2^j:=\{\xi\in [0,1):\,|u_j(t,\xi)-u_i(t,1^-)|\le R\}.
\end{align*}
Since $W_{ij}'$ is continuous thanks to (W2), it is bounded on the interval $[-R,R]$. The $C_{ij}$-convexity of $W$ on $(-\infty,-R)$ yields
\begin{align*}
W_{ij}'(z)-C_{ij}z&\le W'_{ij}(-R)-C_{ij}(-R)\qquad \forall z<-R,
\end{align*}
which can be rewritten as follows using (W3):
\begin{align*}
W_{ij}'(z)&\le C_{ij}z+C_{ij}R-W_{ij}'(R)\qquad \forall z<-R.
\end{align*}
Hence, we obtain
\begin{align*}
\partial_t u_i(t,1^-)&\le \sum_{j=1}^n \int_{A_1}m_iC_{ij}p_j(u_j(t,\xi)-u_i(t,1^-))\dd \xi+C_0,
\end{align*}
for some constant $C_0>0$. Then, with the help of H\"older's inequality,
\begin{align*}
&\sum_{j=1}^n \int_{A_1^j}m_iC_{ij}p_j(u_j(t,\xi)-u_i(t,1^-))\dd \xi\\
&\le m_i\sum_{j=1}^nC_{ij}p_j \left[ \int_0^1 (u_j(t,\xi)-u_i(t,1^-))\dd \xi-\int_{A_2^j}(u_j(t,\xi)-u_i(t,1^-))\dd \xi\right]\\
&\le -m_i\sum_{j=1}^nC_{ij}p_j u_i(t,1^-)+C'\sum_{j=1}^n \int_0^1 p_ju_j(t,\xi)^2\dd \xi+C_1=-m_i\sum_{j=1}^nC_{ij}p_j u_i(t,1^-)+C'\sum_{j=1}^n \mom{\mu_j(t)}+C_1,
\end{align*}
for some constants $C',C_1>0$. We now employ step 1 and observe that, as in Proposition \ref{rem:necconv}, we have $\sum_{j=1}^nC_{ij}p_j\ge \tilde\lambda_0>0$:
\begin{align*}
\partial_t u_i(t,1^-)&\le -m_i\tilde \lambda_0 u_i(t,1^-)+C'',
\end{align*}
for $C''>0$. Gronwall's lemma yields -- thanks to $u_i(0,1^-)<\infty$ -- the existence of a constant $K>0$ such that $\max\limits_{j\in \{1,\ldots,n\}} u_j(t,1^-)\le K$ for all $t\ge 0$.

In analogy, we now consider those $i\in\{1,\ldots,n\}$ such that
\begin{align*}
u_i(t,0^+)&\le \min_{j\in\{1,\ldots,n\}}u_j(t,0^+)+R,
\end{align*}
yielding for each $j\in\{1,\ldots,n\}$ a partition $[0,1)=B_1^j\cup B_2^j$ with
\begin{align*}
B_1^j:=\{\xi\in[0,1):\,u_j(t,\xi)-u_i(t,0^+)>R\},\qquad B_2^j:=\{\xi\in [0,1):\,|u_j(t,\xi)-u_i(t,0^+)|\le R\}.
\end{align*}
Similarly as above, using the symmetry property (W3), we get
\begin{align*}
-\partial_t u_i(t,0^+)&\le -m_i\sum_{j=1}^nC_{ij}p_j (-u_i(t,0^+))+C'\sum_{j=1}^n \mom{\mu_j(t)}+C_1\le -m_i\tilde\lambda_0(-u_i(t,0^+))+C'',
\end{align*}
allowing to proceed as before.\\ Putting the bounds together finishes the proof: $\sup\limits_{t\ge 0}\max\limits_{j\in\{1,\ldots,n\}}\|u_j(t,\cdot)\|_{L^\infty([0,1])}\le K$.
\end{proof}
We thus know, given a confining potential, that the solution lives on a fixed compact interval. It is now a natural question to ask if, for absolutely continuous initial conditions, partial or total collapse of the support can occur in \emph{finite} time. This question is addressed in the following
\begin{prop}[Exclusion of finite-time blow-up]
\label{prop:noblowup}
Let $i\in\{1,\ldots,n\}$ be fixed, but arbitrary. Assume that for all $j\in\{1,\ldots,n\}$ the maps $W_{ij}'$ are Lipschitz-continuous. Suppose moreover that $\supp \mu^0_i$ is a (possibly unbounded) interval and $\mu^0_i$ is absolutely continuous w.r.t. the Lebesgue measure. Assume that its Lebesgue density is continuous on the interior of $\supp\mu^0_i$ and globally bounded. Then, $\mu_i(t)$ is absolutely continuous for all $t\ge 0$.
\end{prop}
\begin{proof}
Our method of proof is an adaptation of the proof of \cite[Thm. 2.9]{burger2008} to the situation at hand. We show that for all $t\ge 0$, there exists $\gamma(t)>0$ such that for all $z\in[0,1)$ and all $h>0$ with $z+h<1$:
\begin{align}
\label{eq:stinc}
\frac1{h}(u_i(t,z+h)-u_i(t,z))&\ge \gamma(t)>0.
\end{align}
That is, $u_i(t,\cdot)$ is \emph{strictly} increasing at each $t\ge 0$. The assumptions on the initial datum above ensure that \eqref{eq:stinc} is true at $t=0$ with some $\gamma(0)>0$. If \eqref{eq:stinc} holds at a given $t_0$, the cumulative distribution function $F_i(t_0,\cdot)$ is Lipschitz-continuous, which implies absolute continuity of $\mu_i(t_0)$.

From \eqref{eq:u}, we get
\begin{align*}
\partial_t (u_i(t,z+h)-u_i(t,z))&=m_i\sum_{j=1}^n p_j\int_0^1\left[W_{ij}'(u_j(t,\xi)-u_i(t,z+h))-W_{ij}'(u_j(t,\xi)-u_i(t,z))\right]\dd\xi.
\end{align*}
Denote by $L_{ij}>0$ the Lipschitz constant of $W_{ij}'$. From the monotonicity $u_i(t,z+h)-u_i(t,z)\ge 0$, it follows that
\begin{align*}
\partial_t (u_i(t,z+h)-u_i(t,z))& \ge -m_i\sum_{j=1}^nL_{ij}p_j (u_i(t,z+h)-u_i(t,z)).
\end{align*}
We subsequently obtain for $\tilde C_i:=m_i\sum\limits_{j=1}^nL_{ij}p_j$ that $\partial_t [(u_i(t,z+h)-u_i(t,z))e^{\tilde C_i t}]\ge 0$, and hence 
\begin{align*}
\frac1{h}(u_i(t,z+h)-u_i(t,z))&\ge \frac1{h}e^{-\tilde C_i t}(u_i(0,z+h)-u_i(0,z))\ge e^{-\tilde C_i t}\gamma(0)>0.
\end{align*}
Letting $\gamma(t):=e^{-\tilde C_i t}\gamma(0)$, \eqref{eq:stinc} follows.
\end{proof}
Naturally, the above result does not extend to $t\to\infty$ since e.g. in the uniformly geodesically convex case, the solution collapses to a Dirac measure in the large-time limit.
\subsection{Long-time behaviour}\label{subsec:long}
We now analyse the long-time behaviour of the solution to \eqref{eq:pdend} in the non-uniformly convex case. 
\begin{thm}[Long-time behaviour]
\label{prop:ltb}
Assume that the solution $\mu$ to \eqref{eq:pdend} is \emph{uniformly confined}, i.e. there exists $K>0$ such that $\supp \mu_i(t)\subset [-K,K]$ holds for all $t\ge 0$ and all $i\in\{1,\ldots,n\}$ as in \eqref{eq:confsupp}. Moreover, suppose that the maps $W_{ij}'$ are Lipschitz-continuous on the interval $[-2K,2K]$ for all $i,j\in\{1,\ldots,n\}$. Set, for $t\ge 0$, $\ent^t:=\ent(\mu(t))$. The following holds:
\begin{enumerate}[(a)]
\item There exists $\ent^\infty\in\R$ such that
\begin{align}
\label{eq:entlim}
\lim_{t\to\infty} \ent^t&=\ent^\infty,
\end{align}
and
\begin{align}
\label{eq:disslim}
\lim_{t\to\infty}\left(\fder{t}\ent^t\right)&=0.
\end{align}
\item For each sequence $(t_k)_{k\in\N}$ in $(0,\infty)$ with $t_k\to\infty$ as $k\to\infty$, there exists a subsequence $(t_{k_l})_{l\in\N}$ and a steady state $\overline{\mu}\in\prb$ of \eqref{eq:pdend} such that for all $i\in\{1,\ldots,n\}$:
\begin{align}
\label{eq:W1c}
\lim_{l\to\infty}\W_1(\mu_i(t_{k_l}),\overline{\mu}_i)=0.
\end{align}
\end{enumerate}
Thus, the $\omega$-limit set of the dynamical system associated to \eqref{eq:pdend} can only contain steady states of \eqref{eq:pdend}.
\end{thm}
\begin{proof}
We proceed similarly to the proof of \cite[Prop. 1]{raoul2012} and observe that along the solution $\mu$, the dissipation of $\ent$ reads
\begin{align}
\label{eq:d1n}
\begin{split}
\fder{t}\ent^t&=-\sum_{i=1}^n m_i\int_\R\left(\sum_{j=1}^n W_{ij}'\ast\mu_j(t)\right)^2\dd\mu_i(t)\\&=-\sum_{i=1}^nm_ip_i\int_0^1\left(\sum_{j=1}^np_j\int_0^1 W_{ij}'(u_i(t,z)-u_j(t,\xi))\dd\xi\right)^2\dd z,
\end{split}
\end{align}
which is non-positive. By \eqref{eq:confsupp}, all $u_i$ are bounded in time and space by the constant $K$. Since $W_{ij}'$ is Lipschitz-continuous, it is differentiable almost everywhere on $[-2K,2K]$. So, another differentiation of the dissipation w.r.t. $t$ shows
\begin{align*}
\frac{\dd^2}{\dd t^2}\ent^t&=-2\sum_{i=1}^nm_ip_i\int_0^1 \left(\sum_{j=1}^n\int_0^1p_jW_{ij}'(u_i(t,z)-u_j(t,\zeta))\dd\zeta\right)\\
&\quad\cdot\left(\sum_{j=1}^np_j\int_0^1W_{ij}''(u_i(t,z)-u_j(t,\xi))\right.\\&\quad\qquad\cdot\left.\left[\sum_{k=1}^n m_ip_k\int_0^1 W_{ik}'(u_k(t,\zeta)-u_i(t,z))\dd\zeta-\sum_{k=1}^n m_jp_k \int_0^1W_{jk}'(u_k(t,\zeta)-u_j(t,\xi))\dd \zeta\right]\dd\xi\right)\dd z.
\end{align*}
By elementary estimates, using in particular that $|W_{ij}''(z)|\le L_{ij}$ a.e. on $[-2K,2K]$ by Lipschitz-continuity, we find
\begin{align}
\label{eq:d2bd}
\sup_{t\ge 0}\left|\frac{\dd^2}{\dd t^2}\ent^t\right|&\le C_2,
\end{align}
for some $C_2>0$. Furthermore, it is easy to conclude from (W2) and \eqref{eq:confsupp} that 
\begin{align}
\label{eq:d0bd}
\inf_{t\ge 0}\ent^t &\ge -C_0,
\end{align}
for another constant $C_0>0$. Putting \eqref{eq:d1n} and \eqref{eq:d0bd} together yields the existence of $\ent^\infty\in\R$ such that \eqref{eq:entlim} holds. We now use \eqref{eq:d2bd} to prove \eqref{eq:disslim}:

Define for $t>\frac{2}{C_2}\sqrt{\ent^0-\ent^\infty}>0$ the quantity $\tau(t):=\frac1{C_2}\sqrt{\ent^{t/2}-\ent^\infty}>0$. Since $\ent^t$ is nonincreasing, we also have $\tau(t)<\frac{t}{2}$. Moreover,
\begin{align*}
\fder{t}\ent^t&=\frac1{\tau(t)}(\ent^t-\ent^\infty)-\frac1{\tau(t)}(\ent^{\tau(t)}-\ent^\infty)+\frac1{\tau(t)}\int_{t-\tau(t)}^t\int_s^t\frac{\dd^2}{\dd\sigma^2}\ent^\sigma \dd\sigma\dd s,
\end{align*}
from which with \eqref{eq:d2bd} and \eqref{eq:entlim} the desired result \eqref{eq:disslim} follows:
\begin{align*}
\left|\fder{t}\ent^t\right|&\le \frac2{\tau(t)}(\ent^{t/2}-\ent^\infty)+\frac1{\tau(t)}C_2\tau(t)^2=(2C_2+1)\sqrt{\ent^{t/2}-\ent^\infty}\stackrel{t\to\infty}{\longrightarrow}0.
\end{align*}
For part (b), let a sequence of time points $t_k\to\infty$ be given. The family $u_i(t_k,\cdot)$ ($k\in\N$; $i\in\{1,\ldots,n\}$) of nondecreasing functions is uniformly bounded in $L^\infty([0,1])$ (by the constant $K$) -- hence bounded in the space $BV([0,1])$. Thus, there exist a subsequence $(t_{k_l})_{l\in\N}$ and nondecreasing maps $\overline{u}_1,\ldots,\overline{u}_n$ with $\|\overline{u}_i\|_{L^\infty([0,1])}\le K$ for all $i\in\{1,\ldots,n\}$, such that $u_i(t_{k_l},\cdot)$ converges to $\overline{u}_i$ in $L^1([0,1])$ and almost everywhere on $[0,1]$, as $l\to\infty$ (for details, see e.g. \cite{ambrosio2000}). The corresponding measure $\overline{\mu}$ belongs to $\prb$ thanks to the dominated convergence theorem. It remains to show that $\overline{\mu}$ is a steady state of system \eqref{eq:pdend}. Define
\begin{align*}
\omega:=-\sum_{i=1}^nm_ip_i\int_0^1 \left(\sum_{j=1}^n p_j \int_0^1 W_{ij}'(\overline{u}_i(z)-\overline{u}_j(\xi))\dd\xi\right)^2\dd z.
\end{align*}
Elementary calculations -- using e.g. Lemma \ref{lem:grad} -- show
\begin{align*}
\left|\fder{t}\ent^{t_{k_l}}-\omega\right|&\le C_0\sum_{i=1}^n\sum_{j=1}^n\int_0^1\int_0^1\left|W_{ij}'(u_i(t_{k_l},z)-u_j(t_{k_l},\xi))-W_{ij}'(\overline{u}_i(z)-\overline{u}_j(\xi))\right|\dd\xi\dd z,
\end{align*}
for a suitable constant $C_0>0$. The Lipschitz-continuity of the $W_{ij}'$ on $[-2K,2K]$, the triangle inequality and \eqref{eq:W1c} then imply
\begin{align*}
\left|\fder{t}\ent^{t_{k_l}}-\omega\right|&\le C_0\sum_{i=1}^n\sum_{j=1}^n L_{ij}\big(\|u_i(t_{k_l},\cdot)-\overline{u}_i\|_{L^1([0,1])}+\|u_j(t_{k_l},\cdot)-\overline{u}_j\|_{L^1([0,1])}\big)\stackrel{l\to\infty}{\longrightarrow}0.
\end{align*}
Hence, because of \eqref{eq:disslim}, $\omega=0$. Specifically, this means that for each $i\in\{1,\ldots,n\}$ and almost every $x\in \supp\overline{\mu}_i$, the following holds:
\begin{align*}
\sum_{j=1}^n\int_\R W_{ij}'(x-y)\dd\overline{\mu}_j(y)&=0.
\end{align*}
So, $\overline{\mu}$ is a solution to \eqref{eq:pdend} and the proof is complete.
\end{proof}
\begin{remark}
The result of Theorem \ref{prop:ltb} does neither yield uniqueness of steady states of \eqref{eq:pdend} nor convergence of the entire curve $\mu$ to some specific object as $t\to\infty$. If there only exists the trivial steady state $\mu^\infty$ from Corollary \ref{coro:posconv} in the set of those elements from $\prb$ with support contained in $[-K,K]$, then Theorem \ref{prop:ltb} implies 
\begin{align*}
\lim_{t\to\infty}\W_1(\mu_i(t),\mu^\infty_i)=0\qquad\forall i\in\{1,\ldots,n\},
\end{align*}
without obtaining any specific rate of convergence.
\end{remark}
If the potential is not confining, convergence may not occur. However, we might observe a \emph{$\delta$-separation} phenomenon: The support of each component $\mu_i$ collapses to a single (but not necessarily fixed) point as $t\to\infty$.
\begin{prop}[$\delta$-separation]\label{prop:delta}
Let $i\in\{1,\ldots,n\}$ be fixed, but arbitrary. Assume that the support of $\mu_i^0$ is compact and that $S_i:=\sum\limits_{j=1}^n\kappa_{ij}p_j>0$ holds. Then,
\begin{align*}
\diam\supp\mu_i(t)&\le e^{-m_iS_i t}\diam\supp\mu_i^0.
\end{align*}
That is, the support of $\mu_i$ contracts at exponential speed.
\end{prop}
\begin{proof}
Recall that $\diam\supp\mu_i(t)=u_i(t,1^-)-u_i(t,0^+)$. We have
\begin{align*}
&\partial_t(u_i(t,1^-)-u_i(t,0^+))=\sum_{j=1}^n m_ip_j\int_0^1\left[W_{ij}'(u_j(t,\xi)-u_i(t,1^-))-W_{ij}'(u_j(t,\xi)-u_i(t,0^+))\right]\dd\xi\\
&\le \sum_{j=1}^n m_ip_j\kappa_{ij}\int_0^1\left[(u_j(t,\xi)-u_i(t,1^-))-(u_j(t,\xi)-u_i(t,0^+))\right]\dd\xi=-m_i S_i (u_i(t,1^-)-u_i(t,0^+)),
\end{align*}
the second-to-last step being a consequence of $\kappa_{ij}$-convexity (W5). Applying Gronwall's lemma completes the proof.
\end{proof}
In the regime where Proposition \ref{prop:delta} is applicable for all $i\in\{1,\ldots,n\}$, system \eqref{eq:pdend} behaves asymptotically as $t\to\infty$ like the particle system \eqref{eq:particle} in the case of only one (heavy) particle for each component ($N_i=1$ for all $i$). Obviously, by Proposition \ref{rem:necconv}, the condition $S_i>0$ above is met in the scenario with uniformly geodesically convex energy. This enables us to improve the convergence result from Section \ref{subsec:gfsol} in one spatial dimension for compactly supported initial data:
\begin{prop}[The uniformly convex case in one spatial dimension]
Assume that the criterion for geodesic convexity \eqref{eq:convcond0} yields $\lambda_0>0$ and suppose that $\mu^0$ has compact support. Then, for each $i\in\{1,\ldots,n\}$,
\begin{align*}
\lim_{t\to\infty}\W_\infty(\mu_i(t),\mu^\infty_i)&=0,
\end{align*}
where $\mu^\infty$ is the unique steady state from Corollary \ref{coro:posconv}.
\end{prop}
In view of Corollary \ref{coro:posconv}, we obtain convergence w.r.t. the stronger topology of the $L^\infty$-Wasserstein distance, but lose the exponential rate of convergence.
\begin{proof}
Fix $i\in\{1,\ldots,n\}$. Since $\lambda_0>0$ and $\supp\mu^0_i$ is compact, we know from Corollary \ref{coro:posconv}, Theorem \ref{prop:confsupp} and Proposition \ref{prop:delta} that $\|u_i(t,\cdot)-x^\infty\|_{L^2([0,1])}\to 0$ as $t\to\infty$, $\|u_i(t,\cdot)\|_{L^\infty([0,1])}\le K$ for all $t\ge 0$ and $u_i(t,1^-)-u_i(t,0^+)\to 0$ as $t\to\infty$. Obviously, if $\lim\limits_{t\to\infty}u_i(t,1^-)=x^\infty$ holds, the desired result follows immediately from
\begin{align*}
\|u_i(t,\cdot)-x^\infty\|_{L^\infty([0,1])}=\max(|u_i(t,1^-)-x^\infty|,|u_i(t,0^+)-x^\infty|),
\end{align*} 
since then also $\lim\limits_{t\to\infty}u_i(t,0^+)=x^\infty$ holds. So, assume that $u_i(t,1^-)$ does \emph{not} converge to $x^\infty$ as $t\to\infty$. Then, there exists $\eps>0$ and a sequence $t_k\to\infty$ such that
\begin{align}
\label{eq:notconverge}
|u_i(t_k,1^-)-x^\infty|&\ge \eps\qquad\forall k\in\N.
\end{align}
Thanks to the observations above, there exist a subsequence $(t_{k_l})_{l\in\N}$ and $\omega\in\R$ such that
\begin{align*}
\lim_{l\to\infty} (u_i(t_{k_l},1^-)-x^\infty)&=\omega,~\text{and}~\lim_{l\to\infty} u_i(t_{k_l},z)=x^\infty~\text{for a.e. }z\in(0,1).
\end{align*}
Immediately, it follows that $\lim\limits_{l\to\infty} u_i(t_{k_l},0^+)=x^\infty+\omega$ and consequently $\omega=0$ by monotonicity. But $\omega=0$ is a contradiction to \eqref{eq:notconverge}.
\end{proof}
%


\subsection*{Acknowledgement}
This research has been supported by the German Research Foundation (DFG), SFB TRR 109. The author thanks Daniel Matthes for helpful discussion and comments.

\bibliographystyle{abbrv}
\bibliography{refint}

\begin{thebibliography}{10}

\bibitem{ambrosio2000}
L.~Ambrosio, N.~Fusco, and D.~Pallara.
\newblock {\em Functions of bounded variation and free discontinuity problems}.
\newblock Oxford Mathematical Monographs. The Clarendon Press Oxford University
  Press, New York, 2000.

\bibitem{savare2008}
L.~Ambrosio, N.~Gigli, and G.~Savar{\'e}.
\newblock {\em Gradient flows in metric spaces and in the space of probability
  measures}.
\newblock Lectures in Mathematics ETH Z\"urich. Birkh\"auser Verlag, Basel,
  second edition, 2008.

\bibitem{balague2014}
D.~Balagu{\'e}, J.~A. Carrillo, and Y.~Yao.
\newblock Confinement for repulsive-attractive kernels.
\newblock {\em Discrete Contin. Dyn. Syst. Ser. B}, 19(5):1227--1248, 2014.

\bibitem{benedetto1997}
D.~Benedetto, E.~Caglioti, and M.~Pulvirenti.
\newblock A kinetic equation for granular media.
\newblock {\em RAIRO Mod\'el. Math. Anal. Num\'er.}, 31(5):615--641, 1997.

\bibitem{bertozzi2010}
A.~L. Bertozzi and J.~Brandman.
\newblock Finite-time blow-up of ${L}^\infty$-weak solutions of an aggregation
  equation.
\newblock {\em Communications in Mathematical Sciences}, 8(1):45--65, 03 2010.

\bibitem{bertozzi2009}
A.~L. Bertozzi, J.~A. Carrillo, and T.~Laurent.
\newblock Blow-up in multidimensional aggregation equations with mildly
  singular interaction kernels.
\newblock {\em Nonlinearity}, 22(3):683--710, 2009.

\bibitem{bertozzi2007}
A.~L. Bertozzi and T.~Laurent.
\newblock Finite-time blow-up of solutions of an aggregation equation in
  {$\mathbf{R}^n$}.
\newblock {\em Comm. Math. Phys.}, 274(3):717--735, 2007.

\bibitem{bertozzi2011}
A.~L. Bertozzi, T.~Laurent, and J.~Rosado.
\newblock {$L^p$} theory for the multidimensional aggregation equation.
\newblock {\em Comm. Pure Appl. Math.}, 64(1):45--83, 2011.

\bibitem{biler2009}
P.~Biler, G.~Karch, and P.~Lauren{\c{c}}ot.
\newblock Blowup of solutions to a diffusive aggregation model.
\newblock {\em Nonlinearity}, 22(7):1559--1568, 2009.

\bibitem{bcc2012}
A.~Blanchet, E.~A. Carlen, and J.~A. Carrillo.
\newblock Functional inequalities, thick tails and asymptotics for the critical
  mass {P}atlak-{K}eller-{S}egel model.
\newblock {\em J. Funct. Anal.}, 262(5):2142--2230, 2012.

\bibitem{blanchet2006}
A.~Blanchet, J.~Dolbeault, and B.~Perthame.
\newblock Two-dimensional {K}eller-{S}egel model: optimal critical mass and
  qualitative properties of the solutions.
\newblock {\em Electron. J. Differential Equations}, pages No. 44, 32 pp.
  (electronic), 2006.

\bibitem{bodnar2006}
M.~Bodnar and J.~J.~L. Vel{\'a}zquez.
\newblock An integro-differential equation arising as a limit of individual
  cell-based models.
\newblock {\em J. Differential Equations}, 222(2):341--380, 2006.

\bibitem{burger2007}
M.~Burger, V.~Capasso, and D.~Morale.
\newblock On an aggregation model with long and short range interactions.
\newblock {\em Nonlinear Anal. Real World Appl.}, 8(3):939--958, 2007.

\bibitem{burger2008}
M.~Burger and M.~Di~Francesco.
\newblock Large time behavior of nonlocal aggregation models with nonlinear
  diffusion.
\newblock {\em Netw. Heterog. Media}, 3(4):749--785, 2008.

\bibitem{carrillo2014gr}
J.~A. Ca{\~n}izo, J.~A. Carrillo, and F.~S. Patacchini.
\newblock Existence of compactly supported global minimisers for the
  interaction energy.
\newblock {\em Arch. Ration. Mech. Anal.}, 217(3):1197--1217, 2015.

\bibitem{canizo2011}
J.~A. Ca{\~n}izo, J.~A. Carrillo, and J.~Rosado.
\newblock A well-posedness theory in measures for some kinetic models of
  collective motion.
\newblock {\em Math. Models Methods Appl. Sci.}, 21(3):515--539, 2011.

\bibitem{carrillo2014}
J.~A. Carrillo, D.~Castorina, and B.~Volzone.
\newblock Ground states for diffusion dominated free energies with logarithmic
  interaction.
\newblock {\em SIAM J. Math. Anal.}, 47(1):1--25, 2015.

\bibitem{carrillo2011}
J.~A. Carrillo, M.~Di~Francesco, A.~Figalli, T.~Laurent, and D.~Slep{\v{c}}ev.
\newblock Global-in-time weak measure solutions and finite-time aggregation for
  nonlocal interaction equations.
\newblock {\em Duke Math. J.}, 156(2):229--271, 2011.

\bibitem{carrillo2012}
J.~A. Carrillo, M.~Di~Francesco, A.~Figalli, T.~Laurent, and D.~Slep{\v{c}}ev.
\newblock Confinement in nonlocal interaction equations.
\newblock {\em Nonlinear Anal.}, 75(2):550--558, 2012.

\bibitem{carrillo2009}
J.~A. Carrillo, M.~R. D'Orsogna, and V.~Panferov.
\newblock Double milling in self-propelled swarms from kinetic theory.
\newblock {\em Kinet. Relat. Models}, 2(2):363--378, 2009.

\bibitem{carrillo2014non}
J.~A. Carrillo, S.~Lisini, and E.~Mainini.
\newblock Gradient flows for non-smooth interaction potentials.
\newblock {\em Nonlinear Anal.}, 100:122--147, 2014.

\bibitem{cmv2003}
J.~A. Carrillo, R.~J. McCann, and C.~Villani.
\newblock Kinetic equilibration rates for granular media and related equations:
  entropy dissipation and mass transportation estimates.
\newblock {\em Rev. Mat. Iberoamericana}, 19(3):971--1018, 2003.

\bibitem{cmv2006}
J.~A. Carrillo, R.~J. McCann, and C.~Villani.
\newblock Contractions in the 2-{W}asserstein length space and thermalization
  of granular media.
\newblock {\em Arch. Ration. Mech. Anal.}, 179(2):217--263, 2006.

\bibitem{carrillo2010}
J.~A. Carrillo and J.~Rosado.
\newblock Uniqueness of bounded solutions to aggregation equations by optimal
  transport methods.
\newblock In {\em European {C}ongress of {M}athematics}, pages 3--16. Eur.
  Math. Soc., Z\"urich, 2010.

\bibitem{carrillo2004}
J.~A. Carrillo and G.~Toscani.
\newblock Wasserstein metric and large-time asymptotics of nonlinear diffusion
  equations.
\newblock In {\em New trends in mathematical physics}, pages 234--244. World
  Sci. Publ., Hackensack, NJ, 2004.

\bibitem{colombo2012}
R.~M. Colombo and M.~L{\'e}cureux-Mercier.
\newblock Nonlocal crowd dynamics models for several populations.
\newblock {\em Acta Math. Sci. Ser. B Engl. Ed.}, 32(1):177--196, 2012.

\bibitem{crippa2013}
G.~Crippa and M.~L{\'e}cureux-Mercier.
\newblock Existence and uniqueness of measure solutions for a system of
  continuity equations with non-local flow.
\newblock {\em NoDEA Nonlinear Differential Equations Appl.}, 20(3):523--537,
  2013.

\bibitem{difranc2013}
M.~Di~Francesco and S.~Fagioli.
\newblock Measure solutions for non-local interaction {PDE}s with two species.
\newblock {\em Nonlinearity}, 26(10):2777--2808, 2013.

\bibitem{dif2016}
M.~Di~Francesco and S.~Fagioli.
\newblock A nonlocal swarm model for predators--prey interactions.
\newblock {\em Math. Models Methods Appl. Sci.}, 26(2):319--355, 2016.

\bibitem{evans2010}
L.~C. Evans.
\newblock {\em Partial differential equations}, volume~19 of {\em Graduate
  Studies in Mathematics}.
\newblock American Mathematical Society, Providence, second edition, 2010.

\bibitem{fellner2010}
K.~Fellner and G.~Raoul.
\newblock Stable stationary states of non-local interaction equations.
\newblock {\em Math. Models Methods Appl. Sci.}, 20(12):2267--2291, 2010.

\bibitem{fellner2011}
K.~Fellner and G.~Raoul.
\newblock Stability of stationary states of non-local equations with singular
  interaction potentials.
\newblock {\em Math. Comput. Modelling}, 53(7-8):1436--1450, 2011.

\bibitem{giacomin2000}
G.~Giacomin, J.~L. Lebowitz, and R.~Marra.
\newblock Macroscopic evolution of particle systems with short- and long-range
  interactions.
\newblock {\em Nonlinearity}, 13(6):2143--2162, 2000.

\bibitem{golse2003}
F.~Golse.
\newblock The mean-field limit for the dynamics of large particle systems.
\newblock In {\em Journ\'ees ``\'{E}quations aux {D}\'eriv\'ees
  {P}artielles''}, pages Exp. No. IX, 1--47. Univ. Nantes, Nantes, 2003.

\bibitem{jko1998}
R.~Jordan, D.~Kinderlehrer, and F.~Otto.
\newblock The variational formulation of the {F}okker-{P}lanck equation.
\newblock {\em SIAM J. Math. Anal.}, 29(1):1--17, 1998.

\bibitem{kang2009}
K.~Kang, B.~Perthame, A.~Stevens, and J.~J.~L. Vel{\'a}zquez.
\newblock An integro-differential equation model for alignment and
  orientational aggregation.
\newblock {\em J. Differential Equations}, 246(4):1387--1421, 2009.

\bibitem{kolokonikov2013}
T.~Kolokolnikov, Y.~Huang, and M.~Pavlovski.
\newblock Singular patterns for an aggregation model with a confining
  potential.
\newblock {\em Phys. D}, 260:65--76, 2013.

\bibitem{laurent2007}
T.~Laurent.
\newblock Local and global existence for an aggregation equation.
\newblock {\em Comm. Partial Differential Equations}, 32(10-12):1941--1964,
  2007.

\bibitem{li2004}
H.~Li and G.~Toscani.
\newblock Long-time asymptotics of kinetic models of granular flows.
\newblock {\em Arch. Ration. Mech. Anal.}, 172(3):407--428, 2004.

\bibitem{luckhaus2012}
S.~Luckhaus, Y.~Sugiyama, and J.~J.~L. Vel{\'a}zquez.
\newblock Measure valued solutions of the 2{D} {K}eller-{S}egel system.
\newblock {\em Arch. Ration. Mech. Anal.}, 206(1):31--80, 2012.

\bibitem{mccann1997}
R.~J. McCann.
\newblock A convexity principle for interacting gases.
\newblock {\em Adv. Math.}, 128(1):153--179, 1997.

\bibitem{mogilner1999}
A.~Mogilner and L.~Edelstein-Keshet.
\newblock A non-local model for a swarm.
\newblock {\em J. Math. Biol.}, 38(6):534--570, 1999.

\bibitem{morale2005}
D.~Morale, V.~Capasso, and K.~Oelschl{\"a}ger.
\newblock An interacting particle system modelling aggregation behavior: from
  individuals to populations.
\newblock {\em J. Math. Biol.}, 50(1):49--66, 2005.

\bibitem{otto2001}
F.~Otto.
\newblock The geometry of dissipative evolution equations: the porous medium
  equation.
\newblock {\em Comm. Partial Differential Equations}, 26(1-2):101--174, 2001.

\bibitem{raoul2012}
G.~Raoul.
\newblock Nonlocal interaction equations: stationary states and stability
  analysis.
\newblock {\em Differential Integral Equations}, 25(5-6):417--440, 2012.

\bibitem{sun2012}
H.~Sun, D.~Uminsky, and A.~L. Bertozzi.
\newblock Stability and clustering of self-similar solutions of aggregation
  equations.
\newblock {\em J. Math. Phys.}, 53(11):115610, 18, 2012.

\bibitem{theil2006}
F.~Theil.
\newblock A proof of crystallization in two dimensions.
\newblock {\em Communications in Mathematical Physics}, 262(1):209--236, 2006.

\bibitem{topaz2004}
C.~M. Topaz and A.~L. Bertozzi.
\newblock Swarming patterns in a two-dimensional kinematic model for biological
  groups.
\newblock {\em SIAM J. Appl. Math.}, 65(1):152--174, 2004.

\bibitem{topaz2006}
C.~M. Topaz, A.~L. Bertozzi, and M.~A. Lewis.
\newblock A nonlocal continuum model for biological aggregation.
\newblock {\em Bull. Math. Biol.}, 68(7):1601--1623, 2006.

\bibitem{toscani2000}
G.~Toscani.
\newblock One-dimensional kinetic models of granular flows.
\newblock {\em M2AN Math. Model. Numer. Anal.}, 34(6):1277--1291, 2000.

\bibitem{toscani2006}
G.~Toscani.
\newblock Kinetic models of opinion formation.
\newblock {\em Commun. Math. Sci.}, 4(3):481--496, 2006.

\bibitem{giessen1999}
A.~E. van Giessen and B.~Widom.
\newblock Path dependence of surface–tension scaling in binary mixtures.
\newblock {\em Fluid Phase Equilibria}, 164(1):1--12, 1999.

\bibitem{villani2003}
C.~Villani.
\newblock {\em Topics in optimal transportation}, volume~58 of {\em Graduate
  Studies in Mathematics}.
\newblock American Mathematical Society, Providence, RI, 2003.

\bibitem{zhang2009}
J.~Zhang and D.~Y. Kwok.
\newblock A mean-field free energy lattice {B}oltzmann model for multicomponent
  fluids.
\newblock {\em The European Physical Journal Special Topics}, 171(1):45--53,
  2009.

\bibitem{zinsl2014}
J.~Zinsl and D.~Matthes.
\newblock Transport distances and geodesic convexity for systems of degenerate
  diffusion equations.
\newblock {\em Calc. Var. Partial Differential Equations}, 54(4):3397--3438,
  2015.

\end{thebibliography}

\end{document}